\documentclass{article}

\usepackage[T1]{fontenc}
\usepackage[utf8]{inputenc}

\usepackage{lmodern}

\usepackage{bm}

\usepackage{color}

\usepackage{microtype}

\usepackage{amsmath}

\usepackage{amsthm}

\usepackage{amssymb}

\usepackage[showonlyrefs]{mathtools}

\usepackage{array}

\usepackage{authblk}

\usepackage{geometry}

\usepackage{fancyhdr, blindtext}
\newcommand{\changefont}{\fontsize{7}{9}\selectfont}
\newcolumntype{L}[1]{>{\raggedright\let\newline\\\arraybackslash\hspace{0pt}}m{#1}}
\newcolumntype{C}[1]{>{\centering\let\newline\\\arraybackslash\hspace{0pt}}m{#1}}
\newcolumntype{R}[1]{>{\raggedleft\let\newline\\\arraybackslash\hspace{0pt}}m{#1}}

\usepackage[colorinlistoftodos]{todonotes}

\usepackage{hyperref}
\hypersetup{colorlinks=true, linkcolor=blue, citecolor=blue, urlcolor=blue,bookmarksnumbered,bookmarksdepth=3}

\usepackage{mathtools}
\usepackage[normalem]{ulem} 

\usepackage{amsmath}
\usepackage{amssymb}
\usepackage{mathtools}
\usepackage{mathrsfs}

\newcommand{\cut}[1]{{}}

\newcommand{\EE}{{\mathbb{E}}} 				
\newcommand{\PP}{\mathbb{P}} 				
\newcommand{\RR}{\mathbb{R}}				
\newcommand{\HH}{\mathbb{H}}				
\newcommand{\NN}{\mathbb{N}}				
\newcommand{\CC}{\mathbb{C}}				

\newcommand{\DeclareAutoPairedDelimiter}[3]{%
	\expandafter\DeclarePairedDelimiter\csname Auto\string#1\endcsname{#2}{#3}%
	\begingroup\edef\x{\endgroup
		\noexpand\DeclareRobustCommand{\noexpand#1}{%
			\expandafter\noexpand\csname Auto\string#1\endcsname*}}%
	\x}

\DeclareAutoPairedDelimiter{\p}{(}{)} 					
\DeclareAutoPairedDelimiter{\sp}{[}{]} 					
\DeclareAutoPairedDelimiter{\abs}{|}{|} 					
\DeclareAutoPairedDelimiter{\cp}{\{}{\}} 				
\DeclareAutoPairedDelimiter{\dotp}{\langle}{\rangle} 	
\DeclareAutoPairedDelimiter{\n}{\Vert}{\Vert} 			
\DeclareAutoPairedDelimiter{\cl}{\lceil}{\rceil}

\newcommand{\cD}{{\mathcal{D}}}

\newcommand{\cF}{{\mathcal{F}}}
\newcommand{\cG}{{\mathcal{G}}}
\newcommand{\cH}{{\mathcal{H}}}

\newcommand{\cO}{{\mathcal{O}}}

\newcommand{\cX}{{\mathcal{X}}}
\newcommand{\cY}{{\mathcal{Y}}}

\usepackage{amsmath}
\usepackage{amsthm}
\usepackage{amssymb}
\usepackage{mathtools}
\usepackage{mathrsfs}

\usepackage[nameinlink]{cleveref}

\newcommand{\bc}{\begin{center}}
\newcommand{\ec}{\end{center}}

\newcommand{\bdm}{\begin{displaymath}}
\newcommand{\edm}{\end{displaymath}}

\newcommand{\beq}{\begin{equation}}
\newcommand{\eeq}{\end{equation}}

\newcommand{\bfl}{\begin{flushleft}}
\newcommand{\efl}{\end{flushleft}}

\newcommand{\bt}{\begin{tabbing}}
\newcommand{\et}{\end{tabbing}}

\newcommand{\beqn}{\begin{align}}
\newcommand{\eeqn}{\end{align}}

\newcommand{\beqs}{\begin{align*}} 
\newcommand{\eeqs}{\end{align*}}  

\newtheoremstyle{Fancyplain}
{\topsep}   
{\topsep}   
{\itshape}  
{0pt}       
{\bfseries} 
{}         
{5pt plus 1pt minus 1pt} 
{\thmname{#1} \thmnumber{#2}. \thmnote{\normalfont\bfseries#3.}}

\theoremstyle{Fancyplain}
\newtheorem{thm}{Theorem}
\newtheorem{lem}{Lemma}
\newtheorem{cor}[lem]{Corollary}

\crefname{thm}{Thm.}{Thms.}
\Crefname{thm}{Theorem}{Theorems}
\crefname{lem}{Lem.}{Lems.}
\Crefname{lem}{Lemma}{Lemmas}
\crefname{cor}{Cor.}{Cors.}
\Crefname{cor}{Corollary}{Corollaries}
\crefname{prop}{Prop.}{Props.}
\Crefname{prop}{Proposition}{Propositions}

\newtheoremstyle{Fancydefinition}
{\topsep}   
{\topsep}   
{\normalfont}  
{0pt}       
{\bfseries} 
{}         
{5pt plus 1pt minus 1pt} 
{\thmname{#1} \thmnumber{#2}. \thmnote{\normalfont\bfseries#3.}}

\theoremstyle{Fancydefinition}
\newtheorem{defn}[lem]{Definition}

\newtheorem{rem}{Remark}

\crefname{defn}{Defn.}{Defns.}
\Crefname{defn}{Definition}{Definitions}
\crefname{example}{Ex.}{Exs.}
\Crefname{example}{Example}{Examples}
\crefname{xca}{Ex.}{Exs.}
\Crefname{xca}{Exercise}{Exercises}
\crefname{rem}{Rem.}{Rems.}
\Crefname{rem}{Remark}{Remarks}
\crefname{asmp}{Asmp.}{Asmps.}
\Crefname{asmp}{Assumption}{Assumptions}
\crefname{section}{Sec.}{Secs.}
\Crefname{section}{Section}{Sections}

\numberwithin{equation}{section}
\numberwithin{figure}{section}

\usepackage{tabularx}

\usepackage{float}
\usepackage{algpseudocode,algorithm,algorithmicx}

\usepackage{comment}

\usepackage{csquotes}
\usepackage[american]{babel}
\usepackage[backend=biber,citestyle=authoryear,bibstyle=authortitle,arxiv=abs,url=true,natbib]{biblatex}

\newbibmacro{string+doiurlisbn}[1]{%
  \iffieldundef{url}{%
    \iffieldundef{doi}{%
      \iffieldundef{isbn}{%
        \iffieldundef{issn}{%
          #1%
        }{%
          \href{http://books.google.com/books?vid=ISSN\thefield{issn}}{#1}%
        }%
      }{%
        \href{http://books.google.com/books?vid=ISBN\thefield{isbn}}{#1}%
      }%
    }{%
      \href{http://dx.doi.org/\thefield{doi}}{#1}%
    }%
  }{%
    \href{\thefield{url}}{#1}%
  }%
}

\DeclareFieldFormat{title}{\usebibmacro{string+doiurlisbn}{\mkbibemph{#1}}}
\DeclareFieldFormat[article,incollection,inproceedings]{title}%
    {\usebibmacro{string+doiurlisbn}{\mkbibquote{#1}}}

\DeclareCiteCommand{\cite}
  {\usebibmacro{prenote}}
  {\usebibmacro{citeindex}%
   \printtext[bibhyperref]{\usebibmacro{cite}}}
  {\multicitedelim}
  {\usebibmacro{postnote}}

\DeclareCiteCommand{\parencite}[\mkbibparens]
  {\usebibmacro{prenote}}
  {\usebibmacro{citeindex}%
    \printtext[bibhyperref]{\usebibmacro{cite}}}
  {\multicitedelim}
  {\usebibmacro{postnote}}

\newcommand{\eps}{\epsilon}
\newcommand{\algoNameS}{{\tt A2BCD}}

\begin{document}
\title{\texttt{A2BCD}: An Asynchronous Accelerated Block Coordinate Descent Algorithm With Optimal Complexity\thanks{This work is supported in part by NSF Grant DMS-1720237 and ONR Grant N000141712162.}}

\author{Robert Hannah\thanks{\href{mailto:roberthannah89@math.ucla.edu}{roberthannah89@math.ucla.edu}}}
\author{Fei Feng\thanks{\href{mailto:fei.feng@math.ucla.edu}{fei.feng@math.ucla.edu}}}
\author{Wotao Yin\thanks{\href{mailto:wotaoyin@math.ucla.edu}{wotaoyin@math.ucla.edu}}}
\affil{Department of Mathematics, University of California, Los Angeles, USA}
\renewcommand\Affilfont{\itshape\small}

\date{\today}

\maketitle

\begin{abstract}
In this paper, we propose the \textbf{A}synchronous \textbf{A}ccelerated Nonuniform Randomized \textbf{B}lock \textbf{C}oordinate \textbf{D}escent algorithm ($\algoNameS$), the first asynchronous Nesterov-accelerated algorithm that achieves optimal complexity. This parallel algorithm solves the unconstrained convex minimization problem, using $p$ computing nodes which compute updates to shared solution vectors, in an asynchronous fashion with no central coordination. Nodes in asynchronous algorithms do not wait for updates from other nodes before starting a new iteration, but simply compute updates using the most recent solution information available.  This allows them to complete iterations much faster than traditional ones, especially at scale, by eliminating the costly synchronization penalty of traditional algorithms. 

We first prove that $\algoNameS$ converges linearly to a solution with a fast accelerated rate that matches the recently proposed $\texttt{NU\_ACDM}$, so long as the maximum delay is not too large. Somewhat surprisingly, $\algoNameS$ pays no complexity penalty for using outdated information. We then prove lower complexity bounds for randomized coordinate descent methods, which show that $\algoNameS$ (and hence $\texttt{NU\_ACDM}$) has optimal complexity to within a constant factor. We confirm with numerical experiments that $\algoNameS$ outperforms $\texttt{NU\_ACDM}$, which is the current fastest coordinate descent algorithm, even at small scale. We also derive and analyze a second-order ordinary differential equation, which is the continuous-time limit of our algorithm, and prove it converges linearly to a solution with a similar accelerated rate.
\end{abstract}

\global\long\def\n#1{\left\Vert #1\right\Vert }

\global\long\def\abs#1{\left|#1\right|}

\global\long\def\p#1{\left(#1\right)}

\global\long\def\sp#1{\left[#1\right]}

\global\long\def\cp#1{\text{\ensuremath{\left\{  #1\right\} } }}

\global\long\def\floor#1{\text{\ensuremath{\left\lfloor #1\right\rfloor }}}

\global\long\def\ceil#1{\text{\ensuremath{\left\lceil #1\right\rceil }}}

\global\long\def\dotp#1{\left\langle #1\right\rangle }

\global\long\def\spn#1{\text{span}\left\{  #1\right\}  }

\global\long\def\RR{\mathbb{R}}

\global\long\def\NN{\mathbb{N}}

\global\long\def\EE{\mathbb{E}}

\global\long\def\HH{\mathbb{H}}

\global\long\def\PP{\mathbb{P}}

\global\long\def\VV{\mathbb{V}}

\global\long\def\CC{\mathbb{C}}

\global\long\def\cD{\text{\ensuremath{\mathcal{D}}}}

\global\long\def\cO{\mathcal{O}}

\global\long\def\cF{\mathcal{F}}

\global\long\def\cG{\mathcal{G}}

\global\long\def\cH{\mathcal{H}}

\global\long\def\cX{\text{\ensuremath{\mathcal{X}}}}

\global\long\def\cY{\text{\ensuremath{\mathcal{Y}}}}

\global\long\def\eps{\epsilon}

\global\long\def\del{\delta}

\global\long\def\gam{\text{\ensuremath{\gamma}}}

\global\long\def\seqi#1{\left(#1^{0},#1^{1},#1^{2},\ldots\right)}

\global\long\def\seqf#1#2{\left(#1^{0},#1^{1},#1^{2},\ldots,#1^{#2}\right)}

\global\long\def\sumi#1{\sum_{#1=1}^{\infty}}

\global\long\def\sumf#1#2{\sum_{#1=1}^{#2}}

\global\long\def\sumav#1#2{\frac{1}{#2}\sum_{#1=1}^{#2}}

\global\long\def\prodf#1#2{\prod_{#1=1}^{#2}}

\global\long\def\prodi#1{\prod_{#1=1}^{\infty}}

\section{Introduction\label{sec:Introduction}}
In this paper, we propose and prove the convergence of the \textbf{A}synchronous \textbf{A}ccelerated Nonuniform Randomized \textbf{B}lock \textbf{C}oordinate \textbf{D}escent algorithm ($\algoNameS$), the first asynchronous Nesterov-accelerated algorithm that achieves optimal complexity. No previous attempts have been able to prove a speedup for asynchronous Nesterov acceleration. We aim to find the minimizer $x_{*}$ of the unconstrained minimization problem:
\begin{align}
\min_{x\in\RR^{d}}f\p x & =f\p{x_{\p 1},\ldots,x_{\p n}}\label{eq:def:Problem}
\end{align}
where $f$ is $\sigma$-strongly convex for $\sigma>0$ with $L$-Lipschitz gradient $\nabla f=\p{\nabla_{1}f,\ldots,\nabla_{n}f}$. $x\in\RR^{d}$ is composed of coordinate blocks $x_{\p 1},\ldots,x_{\p n}$. The coordinate blocks of the gradient $\nabla_{i}f$ are assumed $L_{i}$-Lipschitz with respect to the $i$th block. That is, $\forall x,h\in\RR^{d}$:
%
\begin{align}
\n{\nabla_{i}f\p{x+P_{i}h}-\nabla_{i}f\p x} & \leq L_{i}\n h\label{eq:def:Coordinate-smoothness}
\end{align}
where $P_{i}$ is the projection onto the $i$th block of $\RR^{d}$. Let $\bar{L}\triangleq\frac{1}{n}\sum_{i=1}^{n}L_{i}$ be the average block Lipschitz constant. These conditions on $f$ are assumed throughout this whole paper. Our algorithm can be applied to non-strongly convex objectives ($\sigma=0$) or non-smooth objectives using the \emph{black box reduction} techniques proposed in \parencite{Allen-ZhuHazan2016_optimala}.

Coordinate descent methods, in which a chosen coordinate block $i_{k}$ is updated at every iteration, are a popular way to solve \eqref{eq:def:Problem}. Randomized block coordinate descent (\texttt{RBCD}, \parencite{Nesterov2012_efficiency}) updates a uniformly randomly chosen coordinate block $i_{k}$ with a gradient-descent-like step: $x_{k+1}=x_{k}-(1/L_{i_{k}})\nabla_{i_{k}}f(x_{k})$. This algorithm decreases the error $\EE(f\p{x_{k}}-f(x_{*}))$ to $\eps\p{f\p{x_{0}}-f(x_{*})}$ in $K\p{\eps}=\cO(n(\bar{L}/\sigma)\ln(1/\eps))$ iterations. 

Using a series of averaging and extrapolation steps, \emph{accelerated} \texttt{RBCD} \parencite{Nesterov2012_efficiency} improves \texttt{RBCD}'s iteration complexity $K\p{\eps}$ to $\cO(n\sqrt{\bar{L}/\sigma}\ln(1/\eps))$, which leads to much faster convergence when $\frac{\bar{L}}{\sigma}$ is large. This rate is optimal when all $L_i$ are equal \parencite{LanZhou2017_optimal}. Finally, using a special probability distribution for the random block index $i_{k}$, non-uniform accelerated coordinate descent method \parencite{Allen-ZhuQuRichtarikYuan2016_even} (\texttt{NU\_ACDM}) can further decrease the complexity to $\cO(\sum_{i=1}^{n}\sqrt{L_{i}/\sigma}\ln(1/\eps))$, which can be up to $\sqrt{n}$ times faster than accelerated \texttt{RBCD}, since some $L_{i}$ can be significantly smaller than $L$. \texttt{NU\_ACDM} is the current state-of-the-art coordinate descent algorithm for solving \eqref{eq:def:Problem}. 

Our $\algoNameS$ algorithm generalizes \texttt{NU\_ACDM} to the asynchronous-parallel case. We solve \eqref{eq:def:Problem} with a collection of $p$ computing nodes that continually read a shared-access solution vector $y$ into local memory then compute a block gradient $\nabla_{i}f$, which is used to update shared solution vectors $\p{x,y,z}$. Proving convergence in the asynchronous case requires extensive new technical machinery.

A traditional synchronous-parallel implementation is organized into rounds of computation: Every computing node must complete an update in order for the next iteration to begin. However, this synchronization process can be extremely costly, since the lateness of a single node can halt the entire system. This becomes increasingly problematic with scale, as differences in node computing speeds, load balancing, random network delays, and bandwidth constraints mean that a synchronous-parallel solver may spend more time waiting than computing a solution.

Computing nodes in an asynchronous solver do not wait for others to complete and share their updates before starting the next iteration, but simply continue to update the solution vectors with the most recent information available, without any central coordination. This eliminates costly idle time, meaning that asynchronous algorithms can be much faster than traditional ones, since they have much faster iterations. For instance, random network delays cause asynchronous algorithms to complete iterations $\Omega(\ln(p))$ time faster than synchronous algorithms at scale. This and other factors that influence the speed of iterations are discussed in \parencite{HannahYin2017_more}. However, since many iterations may occur between the time that a node reads the solution vector, and the time that its computed update is applied, effectively the solution vector is being updated with outdated information. At iteration $k$, the block gradient $\nabla_{i_{k}}f$ is computed at a \emph{delayed iterate} $\hat{y}_{k}$ defined as:
\begin{align}
\hat{y}_{k} & =\p{y_{\p{k-j\p{k,1}}},\ldots,y_{\p{k-j\p{k,n}}}}\label{eq:def:Delayed-iterate}
\end{align}
for delay parameters $j\p{k,1},\ldots,j\p{k,n}\in\NN$. Here $j\p{k,i}$ denotes how many iterations out of date coordinate block $i$ is at iteration $k$. Different block may be out of date by different amounts, which is known as an \emph{inconsistent read}
. We assume\footnote{This condition can be relaxed however by techniques in \parencite{HannahYin2017_unboundeda,SunHannahYin2017_asynchronousb,PengXuYanYin2017_convergence,HannahYin2017_more}} that $j\p{k,i}\leq\tau$ for some constant $\tau<\infty$.

\textbf{Our results:} In this paper, we prove that $\algoNameS$ attains \texttt{NU\_ACDM}'s state-of-the-art iteration complexity to highest order for solving \eqref{eq:def:Problem}, so long as delays are not too large. Hence we prove that there is no significant complexity penalty, despite the use of outdated information. The proof is very different from that of \parencite{Allen-ZhuQuRichtarikYuan2016_even}, and involves significant technical innovations and formidable complexity related to the analysis of asynchronicity. 

Since asynchronous algorithms have much faster iterations, and $\algoNameS$ needs essentially the same number of epochs as \texttt{NU\_ACDM} to compute a solution of a target accuracy, we expect $\algoNameS$ to be faster than all existing coordinate descent algorithms. We confirm this with computational experiments, comparing \texttt{A2BCD} to \texttt{NU\_ACDM}, which is the current fastest block coordinate descent algorithm. 

We also prove that \texttt{A2BCD} (and hence \texttt{NU\_ACDM}) has optimal complexity to within a constant factor over a fairly general class of randomized block coordinate descent algorithms. We do this by proving that any algorithm $A$ in this class must in general complete at least $\Omega\p{\sum_{i=1}^{n}\sqrt{L_{i}/\sigma}\ln(1/\eps)}$ random gradient evaluations to decrease the error by a factor of $\epsilon$. This extends results in \parencite{LanZhou2017_optimal} to the case where $L_i$ are not all equal, and the algorithm in question can be asynchronous.

These results are significant, because it was an open question whether Nesterov-type acceleration was compatible with asynchronicity. Not only is this possible, but asynchronous algorithms may even attain optimal complexity in this setting.  In light of the above, it also seems plausible that accelerated incremental algorithms for finite sum problems $f\triangleq \sum_{i=1}^m f_i(x)$ can be made asynchronous-parallel, too.  


We also derive a second-order ordinary differential equation (ODE), which is the continuous-time limit of $\algoNameS$. This extends the ODE found in \parencite{SuBoydCandes2014_differential} to an \emph{asynchronous} accelerated algorithm minimizing a \emph{strongly convex} function. We prove this ODE linearly converges to a solution with the same rate as $\algoNameS$'s, without needing to resort to the restarting technique employed in \parencite{SuBoydCandes2014_differential}. We prove this result using techniques that motivate and clarify the our proof strategy of the main result. 

\section{Main results\label{subsec:Main-result}}
We define the \textbf{condition number} $\kappa=L/\sigma$, and let
$\underbar{L}=\min_{i}L_{i}$ be the smallest block Lipschitz constant. We should consider functions $f$ where it is efficient to calculate blocks of the gradient, so that coordinate-wise parallelization is efficient. That is, the function should be ``coordinate friendly'' \parencite{PengWuXuYanYin2016_coordinate}. This turns out to be a rather wide class of algorithms. So for instance, while the $L^2$-regularized empirical risk minimization problem is not coordinate friendly in general, the equivalent dual problem is, and hence can be solved efficiently by $\algoNameS$ (see \parencite{LinLuXiao2014_accelerateda}, and Section \ref{subsec:Numerical-experiments}).

To calculate the $k+1$'th iteration of the algorithm from iteration $k$, we use a block of the gradient $\nabla_{i_{k}}f$. Finally, we assume that the delays $j\p{k,i}$ are independent of the block sequence $i_{k}$, but otherwise arbitrary (This assumption can be relaxed. See \parencite{SunHannahYin2017_asynchronousb,LeblondPedregosaLacoste-Julien2017_asaga,CannelliFacchineiKungurtsevScutari2017_asynchronous}).

\begin{defn}[Asynchronous Accelerated Randomized Block Coordinate Descent ($\algoNameS$)]\label{def:AANRBCD}
Let $f$ be $\sigma$-strongly convex, and let its gradient $\nabla f$ be $L$-Lipschitz with block coordinate Lipschitz parameters $L_{i}$ \eqref{eq:def:Coordinate-smoothness}. Using these parameters, we sample $i_{k}$ in an independent and identically distributed (IID) fashion according to
\begin{align}
\PP\sp{i_{k}=j} & =L_{j}^{1/2}/S,\quad j\in\cp{1,\ldots,n},
\text{ for }S \triangleq\sum_{i=1}\nolimits^{n}L_{i}^{1/2}.\label{eq:Block-probability-distribution}
\end{align}
Let $\tau$ be the maximum asynchronous delay in the parallel system (or an overestimate of this). We define the dimensionless \textbf{asynchronicity parameter} $\psi$, which is proportional to the maximum delay $\tau$, and quantifies how strongly asynchronicity will affect convergence:
\begin{align}
\psi & =9\p{S^{-1/2}\underbar{L}^{-1/2}L^{3/4}\kappa^{1/4}}\times\tau\label{eq:def:psi}
\end{align}
We use the above system parameters and $\psi$ to define the coefficients $\alpha,\beta$, and $\gamma$ via \eqref{eq:def:alpha},\eqref{eq:def:beta}, and \eqref{eq:def:h}.
\begin{align}
\alpha & \triangleq \p{1+\p{1+\psi}\sigma^{-1/2}S}^{-1}\label{eq:def:alpha}\\
\beta & \triangleq 1-\p{1-\psi}\sigma^{1/2}S^{-1}\label{eq:def:beta}\\
h & \triangleq 1-\frac{1}{2}\sigma^{1/2}\underbar{L}^{-1/2}\psi.\label{eq:def:h}
\end{align}
Hence $\algoNameS$ algorithm is hence defined via the iterations: \eqref{eq:Algorithm-y},\eqref{eq:Algorithm-x}, and \eqref{eq:Algorithm-v}:
\begin{align}
y_{k} & =\alpha v_{k}+\p{1-\alpha}x_{k}\label{eq:Algorithm-y},\\
x_{k+1} & =y_{k}-hL_{i_{k}}^{-1}\nabla_{i^{k}}f\p{\hat{y}_{k}}\label{eq:Algorithm-x},\\
v_{k+1} & =\beta v_{k}+\p{1-\beta}y_{k}-\sigma^{-1/2}L_{i_{k}}^{-1/2}\nabla_{i^{k}}f\p{\hat{y}_{k}}.\label{eq:Algorithm-v}
\end{align}

\end{defn}

Here $\sigma$ may be underestimated, and $L,L_{1},\ldots,L_{n}$ may be overestimated if exact values are unavailable. Notice that $x_{k}$ can be eliminated from the above iteration, and the block gradient $\nabla_{i_{k}}f\p{\hat{y}_{k}}$ only needs to be calculated once per iteration. A larger (or overestimated) maximum delay $\tau$ will cause a larger asynchronicity parameter $\psi$, which leads to more conservative step sizes to compensate.

The convergence of this algorithm can be stated in terms of a Lyapunov function that we will define shortly. We first introduce the \textbf{asynchronicity error}, a powerful tool for analyzing asynchronous algorithms used in several recent works \parencite{PengXuYanYin2016_arock,HannahYin2017_unboundeda,SunHannahYin2017_asynchronousb,HannahYin2017_more}. This error is a weighted sum of the history of the algorithm, with the weights $c_{i}$ decreasing as one goes further back in time. This error appears naturally in the analysis. Much like a well-chosen basis in linear algorithm, it appears to be a natural quantity to consider when analyzing convergence of asynchronous algorithms.

\begin{defn}[Asynchronicity error]\label{def:Asynchronicity-error}
 Using the above parameters, we define:
\begin{align}
A_{k} & =\sumf j{\tau}c_{j}\n{y_{k+1-j}-y_{k-j}}^{2}\label{eq:def:Asynchronicity-error-and-ci}\\
\text{ for }c_{i} & =\frac{6}{S} L^{1/2}\kappa^{3/2}\tau\sum_{j=i}^{\tau}\p{1-\sigma^{1/2}S^{-1}}^{i-j-1}\psi^{-1}.
\end{align}
Here we define $y_{k}=y_{0}$ for all $k<0$.
\end{defn}
The determination of the coefficients $c_{i}$ is in general a very
involved process of trial and error, intuition, and balancing competing
requirements. Obtaining a convergence proof and optimal rates relies
on the skillful choice of these coefficients $c_{i}$. The algorithm
doesn't depend on the coefficients, however; they are only used in
the analysis.

We define $\EE_{k}[X]$ as the expectation of $X$ conditional on conditioned on $\p{x_{0},\ldots,x_{k}}$, $\p{y_{0},\ldots,y_{k}}$, $\p{z_{0},\ldots,z_{k}}$, and $\p{i_{0},\ldots,i_{k-1}}$. To simplify notation\footnote{We can assume $x_{*}=0$ with no loss in generality since we may translate
the coordinate system so that $x_{*}$ is at the origin. We can assume
$f\p{x^{*}}=0$ with no loss in generality, since we can replace $f\p x$
with $f\p x-f\p{x_{*}}$. Without this assumption, the Lyapunov function
simply becomes: $\n{v_{k}-x_{*}}^{2}+A_{k}+c\p{f\p{x_{k}}-f\p{x_{*}}}$.}, we assume that the minimizer $x^{*}=0$, and that $f\p{x^{*}}=0$
with no loss in generality. We
define the \textbf{Lyapunov function}:
\begin{align}
\rho_{k} & =\n{v_{k}}^{2}+A_{k}+cf\p{x_{k}}\label{eq:def:rho}\\
\text{ for }c & =2\sigma^{-1/2}S^{-1}\p{\beta\alpha^{-1}\p{1-\alpha}+1}\label{eq:def:c}.
\end{align}
We define the iteration complexity $K\p{\eps}$ with
respect to some error $E^{k}$ as the number of iterations $K$ such
that the expected error $\EE\sp{E^{K}}$ decreases to less than $\eps E^{0}$.

We now present this paper's first main contribution.

\begin{thm}
\textup{Let $f$ be $\sigma$-strongly convex with a gradient $\nabla f$ that is $L$-Lipschitz with block Lipschitz constants $\cp{L_{i}}_{i=1}^{n}$. Let $\psi$ defined in \eqref{eq:def:psi} satisfy $\psi\leq\frac{3}{7}$ (i.e. $\tau\leq\frac{1}{21}S^{1/2}\underbar{L}^{1/2}L^{-3/4}\kappa^{-1/4}$). Then for $\algoNameS$ we have:\label{thm:Main-theorem}
\begin{align*}
 & \EE_{k}\sp{\n{v_{k+1}}^{2}+A_{k+1}+cf\p{x_{k+1}}} \leq\p{1-\p{1-\psi}\sigma^{1/2}S^{-1}}\p{\n{v_{k}}^{2}+A_{k}+cf\p{x_{k}}}
\end{align*}
As $\sigma^{-1/2}S\to\infty$, the error $\n{v_{k+1}}^{2}+A_{k+1}+cf\p{x_{k+1}}$ has the corresponding complexity $K\p{\eps}=K_{\algoNameS}\p{\eps}$
for:\vspace{-10pt}
\small \begin{align}
K_{\algoNameS}\p{\eps} & =\p{\sum_{i=1}^{n}\sqrt{L_{i}/\sigma}+\cO\p 1}\frac{\ln\p{1/\eps}}{1-\psi}
\end{align} \normalsize
}
\end{thm}

This result is proven in Section \ref{sec:Proof-of-main-theorem}. A  stronger result can be proven, however this adds too much to the complexity of the proof. See Section \ref{sec:Extensions} for a discussion. In practice, asynchronous algorithms are far more resilient to delays than the theory predicts. $\tau$ can be much larger without negatively affecting the convergence rate and complexity. This is perhaps because we are limited to a worst-case analysis, which is not representative of the average-case performance. 

Authors in \parencite{Allen-ZhuQuRichtarikYuan2016_even} (Theorem 5.1) obtain a linear convergence rate of $1-2/\p{1+2\sigma^{-1/2}S}$ for \texttt{NU\_ACDM }which leads to the corresponding iteration complexity of $K_{\text{\text{{\tt NU\_ACDM}}}}\p{\eps}=\p{\sigma^{-1/2}\sum_{i=1}^{n}L_{i}^{1/2}+\cO\p 1}\ln\p{1/\eps}$ as $\sigma^{-1/2}S\to\infty$. And hence, we have:
\begin{align*}
K_{\algoNameS}\p{\eps} & =\frac{1}{1-\psi}\p{1+o\p 1}K_{\text{{\tt NU\_ACDM}}}\p{\eps}
\end{align*}
Hence when $0\leq\psi\ll1$, or equivalently, when $\tau\ll S^{1/2}\underbar{L}^{1/2}L^{-3/4}\kappa^{-1/4}$ the complexity of $\algoNameS$ asymptotically matches that of \texttt{NU\_ACDM}. Hence $\algoNameS$ combines state-of-the-art complexity with the faster iterations and superior scaling that asynchronous iterations allow.

We now present some special cases of the conditions on the maximum
delay $\tau$ required for good complexity.

\begin{cor}
\textup{\label{cor:tau-conditions}Let the conditions of Theorem \ref{thm:Main-theorem} hold. Additionally, let all coordinate-wise Lipschitz constants $L_{i}$ be equal (i.e. $L_{i}=L_{1},\,\forall i$). Then we have $K_{\algoNameS}\p{\eps}\sim K_{\text{\text{{\tt NU\_ACDM}}}}\p{\eps}$ when $\tau \ll n^{1/2}\kappa^{-1/4}\p{L_{1}/L}^{3/4}$.
Further, let us assume all coordinate-wise Lipschitz constants $L_{i}$ equal $L$. Then $K_{\algoNameS}\p{\eps}\sim K_{\text{\text{{\tt NU\_ACDM}}}}\p{\eps}=K_{\text{\text{{\tt ACDM}}}}\p{\eps}$, when $\tau \ll n^{1/2}\kappa^{-1/4}$
}
\end{cor}

\begin{rem}[Reduction to synchronous case]
 Notice that when $\tau=0$, we have $\psi=0$, $c_{i}\equiv0$ and hence $A_{k}\equiv0$. Thus $\algoNameS$ becomes equivalent to \texttt{NU\_ACDM}, the Lyapunov function\footnote{Their Lyapunov function is in fact a generalization of a one found in \parencite{Nesterov2012_efficiency}} $\rho_{k}$ becomes equivalent to one found in \parencite{Allen-ZhuQuRichtarikYuan2016_even}(pg. 9), and Theorem \ref{thm:Main-theorem} yields the same complexity to highest order. 

\end{rem}

\subsection{Optimality}
$\texttt{NU\_ACDM}$ and hence $\texttt{A2BCD}$ are in fact optimal
among a wide class of randomized block gradient algorithms. We
consider a class of algorithms slightly wider than that considered
in \parencite{LanZhou2017_optimal} to encompass asynchronous algorithms. Their result only apply to the case where all $L_i$ are equal in this setting. Our result applies for unequal $L_i$, and potentially asynchronous algorithms. For a subset $S\in\RR^{d}$, we let $\text{IC}\p S$ (inconsistent read) denote the set of vectors $v$ such that $v  =\p{v_{1,1},v_{2,2},\ldots,v_{d,d}}$ for some vectors $v_{1},v_{2},\ldots,v_{d}\in S$. Here $v_{i,j}$
denotes the $j$th component of vector $v_i$. That is, the coordinates
of $v$ are some combination of coordinates of vectors in $S$. Let
$X_{k}=\cp{x_{0},\ldots,x_{k}}$. 

\begin{defn}[Asynchronous Randomized Incremental Algorithms Class]
Consider the unconstrained minimization problem \eqref{eq:def:Problem} for $f$ that is $\sigma$-strongly convex with $L$-Lipschitz gradient, with block-coordinate-wise Lipschitz constants $\cp{L_i}_{i=1}^n$. We define the class $\mathcal{A}$ as algorithms $G$ on this problem such that:

1. For each parameter set $(\sigma,L_{1},\ldots,L_{n},n)$, $G$ has
an associated IID random variable $i_{k}$ with some fixed distribution $\PP\sp{i_{k}}=p_{i}$ for $\sum_{i=1}^n p_i = 1$.

2. For $k\geq0$, the iterates of $A$ satisfy:
\begin{align}
x_{k+1} & \in\text{span}\{\text{IC}\p{X_{k}},\nabla_{i_{0}}f\p{\text{IC}\p{X_{0}}},\nabla_{i_{1}}f\p{\text{IC}\p{X_{1}}},\ldots,\nabla_{i_{k}}f\p{\text{IC}\p{X_{k}}}\}\label{eq:xkp1-in-IC-span}
\end{align}
\end{defn}

This is a rather general class: $x_{k+1}$ can be constructed from any inconsistent reading of past iterates $\text{IC}\p{X_{k}}$, and any past gradient of an inconsistent read $\nabla_{i_{j}}f\p{\text{IC}\p{X_{j}}}$. 

\begin{thm}
For any algorithm $G\in\mathcal{A}$, and parameter set $(\sigma,L_{1},\ldots,L_{n},n)$ for the unconstrained minimization problem, and $k\in\NN$, there is a dimension $d$, a corresponding function $f$ on $\RR^d$, and a starting point $x_{0}$, such that 
\begin{align*}
\EE\n{x_{k}-x_{*}}^{2}/\n{x_{0}-x_{*}}^{2} & \geq \frac{1}{2}\big(1-4/\big(\sum\nolimits_{j=1}^{n}\sqrt{L_{i}/\sigma}+2n\big)\big)^{k}
\end{align*}
Hence the complexity $I\p{\eps}$ for all algorithms in $\mathcal{A}$ satisfies
\begin{align*}
K\p{\eps} & \geq\frac{1}{4}\p{1+o\p 1}\big(\sum\nolimits_{j=1}^{n}\sqrt{L_{i}/\sigma}+2n\big)\ln\p{1/2\eps}
\end{align*}
\end{thm}

Our proof in Section \ref{sec:Optimality-proof} follows very similar lines to \parencite{LanZhou2017_optimal}, which is inspired by Nesterov \parencite{Nesterov2013_introductory}.

\section{ODE Analysis\label{subsec:ODE-Analysis}}
In this section we present and analyze an ODE which is the continuous-time
limit of $\algoNameS$. This ODE is a strongly convex, and asynchronous
version of the ODE found in \parencite{SuBoydCandes2014_differential}.
For simplicity, assume $L_{i}=L,\,\forall i$. We rescale\footnote{I.e. we replace $f(x)$ with $\frac{1}{\sigma}f$.} $f$ so that the strong convexity modulus $\sigma=1$, and hence $\kappa=L/\sigma=L$. Taking the discrete limit of synchronous $\algoNameS$ (i.e. accelerated \texttt{RBCD}), we can derive the following ODE\footnote{For compactness, we have omitted the $(t)$ from time-varying functions $Y(t)$, $\dot{Y}(t)$, $\nabla Y(t)$, etc.} (see Section \eqref{subsec:Derivation-of-ODE}):
\begin{align}
\ddot{Y}+2n^{-1}\kappa^{-1/2}\dot{Y}+2n^{-2}\kappa^{-1}\nabla f\p Y & =0\label{eq:Acceleration-ODE}
\end{align}
We define the parameter $\eta\triangleq n\kappa^{1/2}$, and the energy:
\begin{align}
E\p t & =e^{n^{-1}\kappa^{-1/2}t}\p{f\p{Y}+\frac{1}{4}\n{Y+\eta \dot{Y}}^{2}}\label{eq:ODE-Energy}
\end{align}
This is very similar to the Lyapunov function discussed in \eqref{eq:def:rho},
with $\frac{1}{4}\n{Y\p t+\eta \dot{Y}\p t}^{2}$ fulfilling the role of
$\n{v_{k}}^{2}$, and $A_{k}=0$ (since there is no delay yet). Much
like the traditional analysis in the proof of Theorem \ref{thm:Main-theorem},
we can derive a linear convergence result that has a similar rate. See Section \ref{subsec:Synchronous-ODE-proof} for the proof.

\begin{lem}
If $Y$ satisfies \eqref{eq:Acceleration-ODE}, then the energy $E\p t$
satisfies $E'\p t\leq0$. This implies $E(t)\leq E(0)$, which is equivalent to:
\begin{align*}
f\p{Y\p t}+\frac{1}{4}\n{Y\p t+n\kappa^{1/2}\dot{Y}\p t}^{2} 
\leq &\p{f\p{Y\p 0}+\frac{1}{4}\n{Y\p 0+\eta \dot{Y}\p 0}^{2}}e^{-n^{-1}\kappa^{-1/2}t}
\end{align*}
\end{lem}

We may also analyze an asynchronous version of \eqref{eq:Acceleration-ODE}
to motivate our proof of the main theorem in Section \ref{sec:Proof-of-main-theorem}:
\begin{align}
\ddot{Y}+2n^{-1}\kappa^{-1/2}\dot{Y}+2n^{-2}\kappa^{-1}\nabla f\p{\hat{Y}} & =0\label{eq:Acceleration-ODE-async},
\end{align}
$\hat{Y}\p t$ is a delayed version of $Y\p t$ defined similarly to \eqref{eq:def:Delayed-iterate}, with the delay bounded by $\tau$.

Unfortunately, the energy satisfies (see Section \eqref{subsec:Async-ODE-proof},
\eqref{eq:ODE-non-monotonic-energy}):
\begin{align*}
e^{-\eta^{-1}t}E'\p t & \leq-\frac{1}{8}\eta\n{\dot{Y}}^{2}+3\kappa^{2}\eta^{-1}\tau D\p t,
\text{ for }D\p t \triangleq\int_{t-\tau}^{t}\n{\dot{Y}\p s}^{2}ds.
\end{align*}
Since the right-hand side can be positive, the energy $E(t)$ may not be decreasing in general.
But, we may add a continuous-time version of the
\textbf{asynchronicity error} (see \parencite{SunHannahYin2017_asynchronousb}), much like in Definition \ref{def:Asynchronicity-error},
to create a decreasing energy. Let $c_{0}\geq0$ and $r>0$ be arbitrary constants that will be set later. Define:
\begin{align*}
A\p t & =\int_{t-\tau}^{t}c\p{t-s}\n{\dot{Y}\p s}^{2}ds,
\text{for }c\p t \triangleq c_{0}\p{e^{-rt}+\frac{e^{-r\tau}}{1-e^{-r\tau}}\p{e^{-rt}-1}}.
\end{align*}

\begin{lem}
When $r\tau\leq\frac{1}{2}$, the asynchronicity error $A\p t$ satisfies:
\begin{align*}
e^{-rt}\frac{d}{dt}\p{e^{rt}A\p t} & \leq c_{0}\n{\dot{Y}\p t}^{2}-\frac{1}{2}\tau^{-1}c_{0}D\p t.
\end{align*}
\end{lem}

See Section \ref{subsec:Asynchronicity-error-lemma} for the proof. Adding this error to the Lyapunov function serves a similar purpose
in the continuous-time case as in the proof of Theorem \ref{thm:Main-theorem} (see Lemma \ref{lem:Asynchronicity-error-lemma}).
It allows us to negate $\frac{1}{2}\tau^{-1}c_{0}$ units of $D\p t$
for the cost of creating $c_{0}$ units of $\n{\dot{Y}\p t}^{2}$.

\begin{thm}\label{thm:Async-ODE-Convergence}
Let $c_{0}=6\kappa^{2}\eta^{-1}\tau^{2}$, and $r=\eta^{-1}$. If $\tau\leq\frac{1}{\sqrt{48}}n\kappa^{-1/2}$ then we have:
\begin{align}
e^{-\eta^{-1}t}\frac{d}{dt}\p{E\p t+e^{\eta^{-1}t}A\p t} &\leq 0.
\end{align}
Hence $f\p{Y\p{t}}$ convergence linearly to $f\p{x_*}$ with rate $\cO\p{\exp\p{-tn^{-1}\kappa^{-1/2}}}$

\end{thm}

Notice how this convergence condition is similar to Corollary \ref{cor:tau-conditions}, but a little looser. \vspace{-20pt}

\begin{proof}
\begin{minipage}{0.9\textwidth}
\begin{align*}
 e^{-\eta^{-1}t}\frac{d}{dt}\p{E\p t+e^{\eta^{-1}t}A\p t}
 & \leq\p{c_{0}-\frac{1}{8}\eta}\n{\dot{Y}}^{2}+\p{3\kappa^{2}\eta^{-1}\tau-\frac{1}{2}\tau^{-1}c_{0}}D\p t\\ &=6\eta^{-1}\kappa^{2}\p{\tau^{2}-\frac{1}{48}n^{2}\kappa^{-1}}\n{\dot{Y}}^{2}\leq 0\qedhere
\end{align*}
\end{minipage}
\end{proof}
\vspace{-5pt}
The preceding should act as a guide to understanding the convergence of $\algoNameS$. It may make clearer the use of the Lyapunov function to establish convergence in the synchronous case, the function of the asynchronicity error in the proof of the asynchronous case, and may hopefully elucidate the logic and general strategy of the proof in Section \ref{sec:Proof-of-main-theorem}.

\section{Related work\label{subsec:Related-work}}

We now discuss related work that was not addressed in Section \ref{sec:Introduction}. Nesterov
acceleration is a method for improving an algorithm's iteration complexity's
dependence the condition number
$\kappa$. Nesterov-accelerated methods have been proposed and discovered
in many settings \parencite{Nesterov1983_method,Tseng2008_accelerated,Nesterov2012_efficiency,LinLuXiao2014_accelerateda,LuXiao2014_complexity,Shalev-ShwartzZhang2016_accelerated,Allen-Zhu2017_katyusha},
including for coordinate descent algorithms (algorithms that use 1
gradient block $\nabla_{i}f$ or minimize with respect to $1$ coordinate
block per iteration), and incremental algorithms (algorithms for finite
sum problems $\frac{1}{n}\sum_{i=1}^{n}f_{i}\p x$ that use $1$ function
gradient $\nabla f_{i}\p x$ per iteration). Such algorithms can often
be augmented to solve composite minimization problems (minimization
for objective of the form $f\p x+g\p x$, especially for nonsomooth
$g$), or include constraints.

Asynchronous algorithms were proposed in \parencite{ChazanMiranker1969_chaotic}
to solve linear systems. General convergence results and theory
were developed later in \parencite{Bertsekas1983_distributed,BertsekasTsitsiklis1997_parallel,TsengBertsekasTsitsiklis1990_partially,LuoTseng1992_convergence,LuoTseng1993_convergence,Tseng1991_rate}
for partially and totally asynchronous systems, with essentially-cyclic
block sequence $i_{k}$. More recently, there has been renewed interest
in asynchronous algorithms with random block coordinate updates. Linear
and sublinear convergence results were proven for asynchronous RBCD
\parencite{LiuWright2015_asynchronous,AvronDruinskyGupta2014}, and similar
was proven for asynchronous SGD \parencite{RechtReWrightNiu2011_hogwild},
and variance reduction algorithms \parencite{J.ReddiHefnySraPoczosSmola2015_variance,LeblondPedregosaLacoste-Julien2017_asaga}.

In \parencite{PengXuYanYin2016_arock},
authors proposed and analyzed an asynchronous fixed-point
algorithm called ARock, that takes proximal algorithms, forward-backward,
ADMM, etc. as special cases. \parencite{HannahYin2017_unboundeda,SunHannahYin2017_asynchronousb,PengXuYanYin2017_convergence}
showed that many of the assumptions used in prior work (such as bounded
delay $\tau<\infty$) were unrealistic and unnecessary in general.
In \parencite{HannahYin2017_more} the authors showed that asynchronous
iterations will complete far more iterations per second, and that
a wide class of asynchronous algorithms, including asynchronous $\texttt{RBCD}$,
have the same iteration complexity as their traditional counterparts.
Hence certain asynchronous algorithms can be expected to significantly
outperform traditional ones.

In \parencite{XiaoYuLinChen2017_dscovr} authors propose a novel asynchronous
catalyst-accelerated \parencite{LinMairalHarchaoui2015_universala} primal-dual
algorithmic framework to solve regularized ERM problems. Instead of
using outdated information, they structure the parallel updates so
that the data that an update depends on is up to date (though the
rest of the data may not be). However catalyst acceleration incurs
a logarithmic penalty over Nesterov acceleration in general. Authors in \parencite{FangHuangLin2018_accelerating} skillfully devised accelerated schemes for asynchronous coordinate descent and SVRG using momentum compensation techniques. Although their complexity results have the improved $\sqrt{\kappa}$ dependence on the condition number, they do not prove an asynchronous speedup. Their complexity is $\tau$ times larger than our complexity. Since $\tau$ is necessarily greater than $p$, their results imply that adding more computing nodes will increase running time.

\section{Numerical experiments\label{subsec:Numerical-experiments}}

To investigate the performance of $\texttt{A2BCD}$, we solve the
ridge regression problem. Consider the following primal and corresponding dual objective (see for instance \parencite{LinLuXiao2014_accelerateda}):
\begin{align}
\min_{w\in\RR^{d}}P\p w&=\frac{1}{2n}\n{A^{T}w-l}^{2}+\frac{\lambda}{2}\n w^{2}, \\ \min_{\alpha\in\RR^{n}}D\p{\alpha}&=\frac{1}{2d^2\lambda}\n{A\alpha}^{2}+\frac{1}{2d}\n{\alpha+l}^{2}\label{eq:Ridge-primal-dual}
\end{align}
where $A\in\RR^{d\times n}$ is a matrix of $n$ samples and $d$ features, and $l$
is a label vector. We let $A=\sp{A_1,\ldots,A_m}$ where $A_i$ are the column blocks of $A$. We compare $\texttt{A2BCD}$ (which is asynchronous
accelerated), synchronous $\texttt{NU\_ACDM}$ (which is synchronous
accelerated), and asynchronous $\texttt{RBCD}$ (which is asynchronous
non-accelerated). At each iteration, each node randomly selects a
coordinate block according to \eqref{eq:Block-probability-distribution}, calculates the corresponding block
gradient, and uses it to apply an update to the shared solution vectors.
Nodes in synchronous $\texttt{NU\_ACDM}$ implementation must wait
until all nodes apply an update before they can start the next iteration,
but the asynchronous algorithms simply compute with the most up-to-date
information available. 

We use the datasets $\texttt{w1a}$ (47272 samples, 300 features) and $\texttt{rcv1\_train.binary}$  (20242 samples, 47236 features) from LIBSVM \parencite{ChangLin2011_libsvm}. The algorithm is implemented in a multi-threaded fashion using C++11 and GNU Scientific Library with a shared memory architecture. We use 40 threads on two 2.5GHz 10-core Intel Xeon E5-2670v2 processors. 

To estimate $\psi$, one can first performed a dry run with all coefficient set to $0$ to estimate $\tau$. All function parameters can be calculated exactly for this problem in terms of the data matrix and $\lambda$.  We can then use these parameters and this tau to calculate $\psi$. $\psi$ and $\tau$ merely change the parameters, and do not change execution patterns of the processors. Hence their parameter specification doesn't affect the observed delay. Through simple tuning though, we found that $\psi=0.25$ resulted in good performance.

In tuning for general problems, there are theoretical reasons why it is difficult to attain acceleration without some prior knowledge of $\sigma$, the strong convexity modulus \parencite{Arjevani2017_limitations}. Ideally $\sigma$ is pre-specified for instance in a regularization term. If the Lipschitz constants $L_i$ cannot be calculated directly (which is rarely the case for the classic dual problem of empirical risk minimization objectives), the line-search method discussed in \parencite{RouxSchmidtBach2012_stochastic} Section 4 can be used. 

A critical ingredient in the efficient implementation of $\algoNameS$ and $\texttt{NU\_ACDM}$ for this problem is the efficient update scheme discussed in \parencite{LeeSidford2013_efficient}. In linear regression applications such as this, it is essential to be able to efficiently maintain or recover $Ay$. This is because calculating block gradients requires the vector $A_i^TAy$, and without an efficient way to recover $Ay$, block gradient evaluations are essentially $50\%$ as expensive as full-gradient calculations. Unfortunately, every accelerated iteration results in dense updates to $y^k$ because of the averaging step in \eqref{eq:Algorithm-y}. Hence $Ay$ must be recalculated from scratch. 

However \parencite{LeeSidford2013_efficienta} introduces a linear transformation that allows for an equivalent iteration that results in sparse updates to new iteration variables $p$ and $q$. The original purpose of this transformation was to ensure that the averaging steps (e.g. \eqref{eq:Algorithm-y}) do not dominate the computational cost for sparse problems. However we find a more important secondary use which applies to both sparse and dense problems. Since the updates to $p$ and $q$ are sparse coordinate-block updates, the vectors $Ap$, and $Aq$ can be efficiently maintained, and therefore block gradients can be efficiently calculated. Implementation details are discussed in more detail in Appendix \ref{sec:Efficient implementation}. 

In Table \ref{tab:Error-vs-time}, we plot the sub-optimality vs. time for decreasing values of $\lambda$, which corresponds to increasingly large condition numbers $\kappa$. When $\kappa$ is small, acceleration doesn't result in a significantly better convergence rate, and hence $\texttt{A2BCD}$ and async-RBCD both outperform sync-$\texttt{NU\_ACDM}$ since they complete faster iterations at similar complexity. Acceleration for low $\kappa$ has unnecessary overhead, which means async-RBCD can be quite competitive. When $\kappa$ becomes large, async-RBCD is no longer competitive, since it has a poor convergence rate. We observe that $\texttt{A2BCD}$ and sync-$\texttt{NU\_ACDM}$ have essentially the same convergence rate, but $\texttt{A2BCD}$ is up to $2-3.5\times$ faster than sync-$\texttt{NU\_ACDM}$ because it completes much faster iterations. We observe this advantage despite the fact that we are in an ideal environment for synchronous computation: A small, homogeneous, high-bandwidth, low-latency cluster. In large-scale heterogeneous systems with greater synchronization overhead, bandwidth constraints, and latency, we expect $\texttt{A2BCD}$'s advantage to be much larger.

\begin{table}[H]\label{tab:Error-vs-time}
\begin{tabularx}{\textwidth}{XXX}
\includegraphics[width=0.3\textwidth]{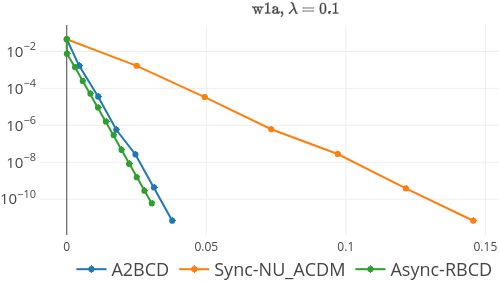}
&
\includegraphics[width=0.3\textwidth]{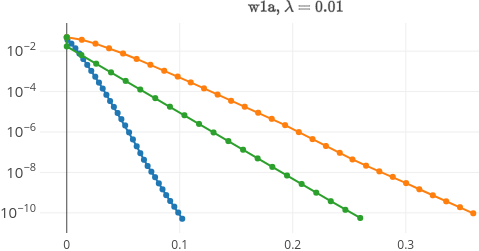}
&
\includegraphics[width=0.3\textwidth]{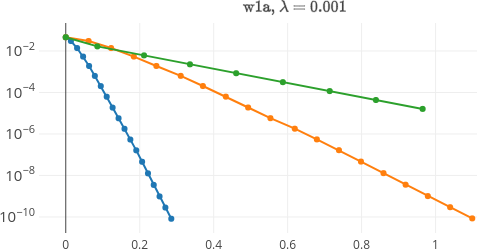}
\\

\includegraphics[width=0.3\textwidth]{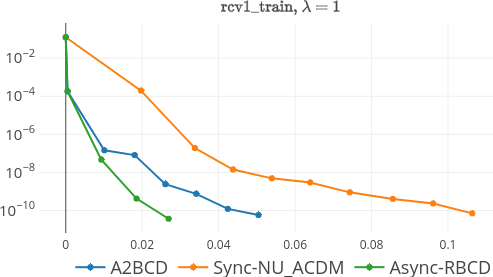}
&
\includegraphics[width=0.3\textwidth]{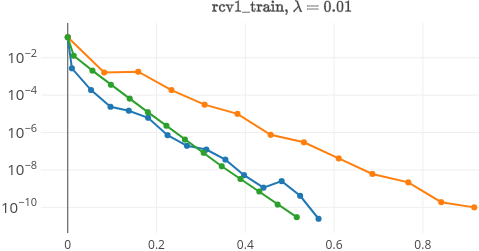}
&
\includegraphics[width=0.3\textwidth]{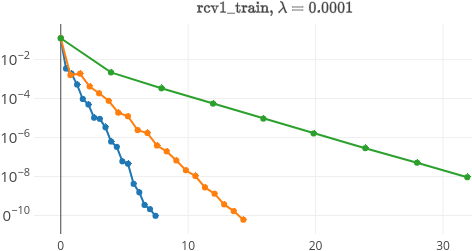}
\end{tabularx}
\caption{Sub-optimality $f(y^k)-f(x^*)$ (y-axis) vs time in seconds (x-axis) for $\texttt{A2BCD}$, synchronous $\texttt{NU\_ACDM}$, and asynchronous RBCD for data sets $\texttt{w1a}$ and $\texttt{rcv1\_train}$ for various values of $\lambda$.}
\end{table}

\pdfbookmark[1]{References}{References}
\printbibliography

\newpage

\appendix

\section{Proof of the main result\label{sec:Proof-of-main-theorem}}
We first recall a couple of inequalities for convex functions.

\begin{lem}
\textup{Let $f$ be $\sigma$-strongly convex with $L$-Lipschitz gradient. Then we have:
\begin{align}
f\p y & \leq f\p x+\dotp{y-x,\nabla f\p x}+\frac{1}{2}L\n{y-x}^{2},\,\forall x,y\label{eq:Lipschitz-convex-bound}\\
f\p y & \geq f\p x+\dotp{y-x,\nabla f\p x}+\frac{1}{2}\sigma\n{y-x}^{2},\,\forall x,y\label{eq:Strongly-convex-bound}
\end{align}
}
\end{lem}

We also find it convenient to define the norm:
\begin{align}
\n s_{*} & =\sqrt{\sum_{i=1}^{n}L_{i}^{-1/2}\n{s_{i}}^{2}}\label{eq:def:Star-norm}
\end{align}

\subsection{Starting point}
First notice that using the definition \eqref{eq:Algorithm-v} of $v_{k+1}$ we have:
\begin{align}
\n{v_{k+1}}^{2} & =\n{\beta v_{k}+\p{1-\beta}y_{k}}^{2}-2\sigma^{-1/2}L_{i_{k}}^{-1/2}\dotp{\beta v_{k}+\p{1-\beta}y_{k},\nabla_{i^{k}}f\p{\hat{y}_{k}}}+\sigma^{-1}L_{i_{k}}^{-1}\n{\nabla_{i^{k}}f\p{\hat{y}_{k}}}^{2}\nonumber \\
\EE_{k}\n{v_{k+1}}^{2} & =\underbrace{\n{\beta v_{k}+\p{1-\beta}y_{k}}^{2}}_{A}-2\sigma^{-1/2}S^{-1}\underbrace{\dotp{\beta v_{k}+\p{1-\beta}y_{k},\nabla f\p{\hat{y}_{k}}}}_{B}\label{eq:Ek-vkp1}\\
 & +S^{-1}\sigma^{-1}\underbrace{\sum_{i=1}^{n}L_{i}^{-1/2}\n{\nabla_{i}f\p{\hat{y}_{k}}}^{2}}_{C}\nonumber 
\end{align}

We have the following general identity:
\begin{align}
\n{\beta x+\p{1-\beta}y}^{2} & =\beta\n x^{2}+\p{1-\beta}\n y^{2}-\beta\p{1-\beta}\n{x-y}^{2},\,\forall x,y\label{eq:Combettes-inequality}
\end{align}
It can also easily be verified from \eqref{eq:Algorithm-y} that we have:
\begin{align}
v_{k} & =y_{k}+\alpha^{-1}\p{1-\alpha}\p{y_{k}-x_{k}}\label{eq:vk-in-terms-xk-yk}
\end{align}
Using \eqref{eq:Combettes-inequality} on term $A$, \eqref{eq:vk-in-terms-xk-yk} on term $B$, and recalling the definition \eqref{eq:def:Star-norm} on term $C$, we have from \eqref{eq:Ek-vkp1}:
\begin{align}
\EE_{k}\n{v_{k+1}}^{2} & =\beta\n{v_{k}}^{2}+\p{1-\beta}\n{y_{k}}^{2}-\beta\p{1-\beta}\n{v_{k}-y_{k}}^{2}+S^{-1}\sigma^{-1/2}\n{\nabla f\p{\hat{y}_{k}}}_{*}^{2}\label{eq:Starting-point}\\
 & -2\sigma^{-1/2}S^{-1}\beta\alpha^{-1}\p{1-\alpha}\dotp{y_{k}-x_{k},\nabla f\p{\hat{y}_{k}}}-2\sigma^{-1/2}S^{-1}\dotp{y_{k},\nabla f\p{\hat{y}_{k}}}\nonumber 
\end{align}

This inequality is our starting point. We analyze the terms on the
second line in the next section.

\subsection{The Cross Term}
To analyze these terms, we need a small lemma. This lemma is fundamental
in allowing us to deal with asynchronicity.

\begin{lem}\label{lem:Async-difference-lemma}
Let $\chi,A>0$. Let the delay be bounded by $\tau$. Then:
\begin{align*}
A\n{\hat{y}_{k}-y_{k}} & \leq\frac{1}{2}\chi^{-1}A^{2}+\frac{1}{2}\chi\tau\sum_{j=1}^{\tau}\n{y_{k+1-j}-y_{k-j}}^{2}
\end{align*}

\end{lem}

\begin{proof}
See \parencite{HannahYin2017_more}.
\end{proof}

\begin{lem}\label{lem:Cross-term-lemma}
We have:
\begin{align}
-\dotp{\nabla f\p{\hat{y}_{k}},y_{k}} & \leq-f\p{y_{k}}-\frac{1}{2}\sigma\p{1-\psi}\n{y_{k}}^{2}+\boldsymbol{\frac{1}{2}L\kappa\psi^{-1}\tau\sum_{j=1}^{\tau}\n{y_{k+1-j}-y_{k-j}}^{2}}\label{eq:Cross-term-lemma-1}\\
\dotp{\nabla f\p{\hat{y}_{k}},x_{k}-y_{k}} & \leq f\p{x_{k}}-f\p{y_{k}}\label{eq:Cross-term-lemma-2}\\
 & +\boldsymbol{\frac{1}{2}L\alpha\p{1-\alpha}^{-1}\p{\kappa^{-1}\psi\beta\n{v_{k}-y_{k}}^{2}+\kappa\psi^{-1}\beta^{-1}\tau\sum_{j=1}^{\tau}\n{y_{k+1-j}-y_{k-j}}^{2}}}\nonumber 
\end{align}

\end{lem}

The terms in bold in \eqref{eq:Cross-term-lemma-1} and \eqref{eq:Cross-term-lemma-2} are a result of the asynchronicity, and are identically $0$ in its absence.

\begin{proof}
Our strategy is to separately analyze terms that appear in the traditional analysis of \parencite{Nesterov2012_efficiency}, and the terms that result from asynchronicity. We first prove \eqref{eq:Cross-term-lemma-1}:
\begin{align}
-\dotp{\nabla f\p{\hat{y}_{k}},y_{k}} & =-\dotp{\nabla f\p{y_{k}},y_{k}}-\dotp{\nabla f\p{\hat{y}_{k}}-\nabla f\p{y_{k}},y_{k}}\nonumber \\
 & \leq-f\p{y_{k}}-\frac{1}{2}\sigma\n{y_{k}}^{2}+L\n{\hat{y}_{k}-y_{k}}\n{y_{k}}\label{eq:Cross-term-proof-1-SC}
\end{align}
\eqref{eq:Cross-term-proof-1-SC} follows from strong convexity  (\eqref{eq:Strongly-convex-bound} with $x=y_{k}$ and $y=x_{*}$), and the fact that $\nabla f$ is $L$-Lipschitz. The term due to asynchronicity becomes:
\begin{align*}
L\n{\hat{y}_{k}-y_{k}}\n{y_{k}} & \leq\frac{1}{2}L\kappa^{-1}\psi\n{y_{k}}^{2}+\frac{1}{2}L\kappa\psi^{-1}\tau\sum_{j=1}^{\tau}\n{y_{k+1-j}-y_{k-j}}^{2}
\end{align*}
using Lemma \ref{lem:Async-difference-lemma} with $\chi=\kappa\psi^{-1}, A=\n{y_k}$. Combining this with \eqref{eq:Cross-term-proof-1-SC} completes the proof of \eqref{eq:Cross-term-lemma-1}.

We now prove \eqref{eq:Cross-term-lemma-2}:
\begin{align*}
&\dotp{\nabla f\p{\hat{y}_{k}},x_{k}-y_{k}}\\
& =\dotp{\nabla f\p{y_{k}},x_{k}-y_{k}}+\dotp{\nabla f\p{\hat{y}_{k}}-\nabla f\p{y_{k}},x_{k}-y_{k}}\\
 & \leq f\p{x_{k}}-f\p{y_{k}}+L\n{\hat{y}_{k}-y_{k}}\n{x_{k}-y_{k}}\\
 & \leq f\p{x_{k}}-f\p{y_{k}}\\
 & +\frac{1}{2}L\p{\kappa^{-1}\psi\beta\alpha^{-1}\p{1-\alpha}\n{x_{k}-y_{k}}^{2}+\kappa\psi^{-1}\beta^{-1}\alpha\p{1-\alpha}^{-1}\tau\sum_{j=1}^{\tau}\n{y_{k+1-j}-y_{k-j}}^{2}}
\end{align*}
Here the last line follows from Lemma \ref{lem:Async-difference-lemma} with $\chi=\kappa\psi^{-1}\beta^{-1}\alpha\p{1-\alpha}^{-1}$, $A=n{x_{k}-y_{k}}$. We can complete the proof using the following identity that can be easily obtained from \eqref{eq:Algorithm-y}: 
\begin{align*}
y_{k}-x_{k} & =\alpha\p{1-\alpha}^{-1}\p{v_{k}-y_{k}}\qedhere
\end{align*}
\end{proof}

\subsection{Function-value term}

Much like \parencite{Nesterov2012_efficiency}, we need a $f\p{x^{k}}$ term in the Lyapunov function (see the middle of page $357$). However we additionally need to consider asynchronicity when analyzing the growth of this term. Again terms due to asynchronicity are emboldened.

\begin{lem}\label{lem:Function-value-lemma}
We have:
\begin{align*}
\EE_{k}f\p{x_{k+1}} & \leq f\p{y_{k}}-\frac{1}{2}h\p{2-h\p{1+\boldsymbol{\frac{1}{2}\sigma^{1/2}\underbar{L}^{-1/2}\psi}}}S^{-1}\n{\nabla f\p{\hat{y}_{k}}}_{*}^{2}\\
 & +\boldsymbol{S^{-1}L\sigma^{1/2}\kappa\psi^{-1}\tau\sum_{j=1}^{\tau}\n{y_{k+1-j}-y_{k-j}}^{2}}
\end{align*}

\end{lem}

\begin{proof}
From the definition
\eqref{eq:Algorithm-x} of $x_{k+1}$, we can see that $x_{k+1}-y_{k}$ is supported on block $i_k$. Since each gradient block $\nabla_i f$ is $L_i$ Lipschitz with respect to changes to block $i$, we can use \eqref{eq:Lipschitz-convex-bound} to obtain:
\begin{align}
f\p{x_{k+1}} & \leq f\p{y_{k}}+\dotp{\nabla f\p{y_{k}},x_{k+1}-y_{k}}+\frac{1}{2}L_{i_k}\n{x_{k+1}-y_{k}}^{2}\nonumber \\
 \text{(from \eqref{eq:Algorithm-x})}& =f\p{y_{k}}-hL_{i_{k}}^{-1}\dotp{\nabla_{i_{k}}f\p{y_{k}},\nabla_{i_{k}}f\p{\hat{y}_{k}}}+\frac{1}{2}h^{2}L_{i_{k}}^{-1}\n{\nabla_{i_{k}}f\p{\hat{y}_{k}}}^{2}\nonumber \\
 & =f\p{y_{k}}-hL_{i_{k}}^{-1}\dotp{\nabla_{i_{k}}f\p{y_{k}}-\nabla_{i_{k}}f\p{\hat{y}_{k}},\nabla_{i^{k}}f\p{\hat{y}_{k}}}-\frac{1}{2}h\p{2-h}L_{i_{k}}^{-1}\n{\nabla_{i_{k}}f\p{\hat{y}_{k}}}^{2}\nonumber \\
\EE_{k}f\p{x_{k+1}} & \leq f\p{y_{k}}-hS^{-1}\sum_{i=1}^{n}L_{i}^{-1/2}\dotp{\nabla_{i}f\p{y_{k}}-\nabla_{i}f\p{\hat{y}_{k}},\nabla_{i}f\p{\hat{y}_{k}}}\label{eq:fn-value-proof-Efkp1}\\
&-\frac{1}{2}h\p{2-h}S^{-1}\n{\nabla f\p{\hat{y}_{k}}}_{*}^{2}\nonumber
\end{align}
Here the last line followed from the definition \eqref{eq:def:Star-norm} of the norm $\n{\cdot}_{*1/2}$. We now analyze the middle term:
\begin{align}
 & -\sum_{i=1}^{n}L_{i}^{-1/2}\dotp{\nabla_{i}f\p{y_{k}}-\nabla_{i}f\p{\hat{y}_{k}},\nabla_{i}f\p{\hat{y}_{k}}}\nonumber \\
 & =-\dotp{\sum_{i=1}^{n}L_{i}^{-1/4}\p{\nabla_{i}f\p{y_{k}}-\nabla_{i}f\p{\hat{y}_{k}}},\sum_{i=1}^{n}L_{i}^{-1/4}\nabla_{i}f\p{\hat{y}_{k}}}\nonumber \\
\text{(Cauchy Schwarz)} & \leq\n{\sum_{i=1}^{n}L_{i}^{-1/4}\p{\nabla_{i}f\p{y_{k}}-\nabla_{i}f\p{\hat{y}_{k}}}}\n{\sum_{i=1}^{n}L_{i}^{-1/4}\nabla_{i}f\p{\hat{y}_{k}}}\nonumber \\
 & =\p{\sum_{i=1}^{n}L_{i}^{-1/2}\n{\nabla_{i}f\p{y_{k}}-\nabla_{i}f\p{\hat{y}_{k}}}^{2}}^{1/2}\p{\sum_{i=1}^{n}L_{i}^{-1/2}\n{\nabla_{i}f\p{\hat{y}_{k}}}^{2}}^{1/2}\nonumber \\
\text{($\underbar{L}\leq L_{i},\forall i$ and \eqref{eq:def:Star-norm})} & \leq\underbar{L}^{-1/4}\n{\nabla f\p{y_{k}}-\nabla f\p{\hat{y}_{k}}}\n{\nabla f\p{\hat{y}_{k}}}_{*}\nonumber \\
 \text{($\nabla f$ is $L$-Lipschitz)}& \leq\underbar{L}^{-1/4}L\n{y_{k}-\hat{y}_{k}}\n{\nabla f\p{\hat{y}_{k}}}_{*}\nonumber
\end{align}
We then apply Lemma \ref{lem:Async-difference-lemma} to this
with $\chi=2h^{-1}\sigma^{1/2}\underbar{L}^{1/4}\kappa\psi^{-1}, A=\n{\nabla f\p{\hat{y}_{k}}}_{*}$ to yield:
\begin{align}
-\sum_{i=1}^{n}L_{i}^{-1/2}\dotp{\nabla_{i}f\p{y_{k}}-\nabla_{i}f\p{\hat{y}_{k}},\nabla_{i}f\p{\hat{y}_{k}}} &\leq h^{-1}L\sigma^{1/2}\kappa\psi^{-1}\tau\sum_{j=1}^{\tau}\n{y_{k+1-j}-y_{k-j}}^{2}\label{eq:fn-value-proof-gradient-inner-product}\\
&+\frac{1}{4}h\sigma^{1/2}\underbar{L}^{-1/2}\psi\n{\nabla f\p{\hat{y}_{k}}}_{*}^{2}\nonumber
\end{align}
Finally to complete the proof, we combine \eqref{eq:fn-value-proof-Efkp1}, with
\eqref{eq:fn-value-proof-gradient-inner-product}.
\end{proof}

\subsection{Asynchronicity error}
The previous inequalities produced difference terms of the form $\n{y_{k+1-j}-y_{k-j}}^{2}$. The following lemma shows how these errors can be incorporated into a Lyapunov function.

\begin{lem}
Let $0<r<1$ and consider the asynchronicity error and corresponding coefficients:\label{lem:Asynchronicity-error-lemma}
\begin{align*}
A_{k} & =\sumi jc_{j}\n{y_{k+1-j}-y_{k-j}}^{2}\\
c_{i} & =\sum_{j=i}^{\infty}r^{i-j-1}s_{j}
\end{align*}
This sum satisfies: 
\begin{align*}
\EE_{k}\sp{A_{k+1}-rA_{k}} & =c_{1}\EE_{k}\n{y_{k+1}-y_{k}}^{2}-\sum_{j=1}^{\infty}s_{j}\n{y_{k+1-j}-y_{k-j}}^{2}
\end{align*}

\end{lem}

\begin{rem}[Interpretation]
This result means that an asynchronicity error term $A_{k}$ can negate a series of difference terms $-\sum_{j=1}^{\infty}s_{j}\n{y_{k+1-j}-y_{k-j}}^{2}$ at the cost of producing an additional error $c_{1}\EE_{k}\n{y_{k+1}-y_{k}}^{2}$, while maintaining a convergence rate of $r$. This essentially converts difference terms, which are hard to deal with, into a $\n{y_{k+1}-y_{k}}^{2}$ term which can be negated by other terms in the Lyapunov function.
The proof is straightforward.

\end{rem}

\begin{proof}
\begin{align*}
\EE_{k}\sp{A_{k+1}-rA_{k}} & =\EE_{k}\sum_{j=0}^{\infty}c_{j+1}\n{y_{k+1-j}-y_{k-j}}^{2}-r\EE_{k}\sumi jc_{j}\n{y_{k+1-j}-y_{k-j}}^{2}\\
 & =c_{1}\EE_{k}\n{y_{k+1}-y_{k}}^{2}+\EE_{k}\sumi j\p{c_{j+1}-rc_{j}}\n{y_{k+1-j}-y_{k-j}}^{2}
\end{align*}
Noting the following completes the proof:
\begin{align*}
c_{i+1}-rc_{i} & =\sum_{j=i+1}^{\infty}r^{i+1-j-1}s_{j}-r\sum_{j=i}^{\infty}r^{i-j-1}s_{j}=-s_{i}\qedhere
\end{align*}
\end{proof}

Given that $A_{k}$ allows us to negate difference terms, we now analyze the cost $c_{1}\EE_{k}\n{y_{k+1}-y_{k}}^{2}$ of this negation.

\begin{lem}
We have:\label{lem:y-Diff-lemma}
\begin{align*}
\EE_{k}\n{y_{k+1}-y_{k}}^{2} & \leq2\alpha^{2}\beta^{2}\n{v_{k}-y_{k}}^{2}+2S^{-1}\underbar{L}^{-1}\n{\nabla f\p{\hat{y}_{k}}}^{2}
\end{align*}

\end{lem}

\begin{proof}
\begin{align}
y_{k+1}-y_{k} & =\p{\alpha v_{k+1}+\p{1-\alpha}x_{k+1}}-y_{k}\nonumber \\
 & =\alpha\p{\beta v_{k}+\p{1-\beta}y_{k}-\sigma^{-1/2}L_{i_{k}}^{-1/2}\nabla_{i_{k}}f\p{\hat{y}_{k}}}+\p{1-\alpha}\p{y_{k}-hL_{i_{k}}^{-1}\nabla_{i^{k}}f\p{\hat{y}_{k}}}-y_{k}\label{eq:Ydiff-proof-vsub}\\
 & =\alpha\beta v_{k}+\alpha\p{1-\beta}y_{k}-\alpha\sigma^{-1/2}L_{i_{k}}^{-1/2}\nabla_{i_{k}}f\p{\hat{y}_{k}}-\alpha y_{k}-\p{1-\alpha}hL_{i_{k}}^{-1}\nabla_{i^{k}}f\p{\hat{y}_{k}}\nonumber \\
 & =\alpha\beta\p{v_{k}-y_{k}}-\p{\alpha\sigma^{-1/2}L_{i_{k}}^{-1/2}+h\p{1-\alpha}L_{i_{k}}^{-1}}\nabla_{i^{k}}f\p{\hat{y}_{k}}\nonumber \\
\n{y_{k+1}-y_{k}}^{2} & \leq2\alpha^{2}\beta^{2}\n{v_{k}-y_{k}}^{2}+2\p{\alpha\sigma^{-1/2}L_{i_{k}}^{-1/2}+h\p{1-\alpha}L_{i_{k}}^{-1}}^{2}\n{\nabla_{i^{k}}f\p{\hat{y}_{k}}}^{2}\label{eq:Ydiff-proof-split}
\end{align}
Here \eqref{eq:Ydiff-proof-vsub} following from \eqref{eq:Algorithm-v}, the definition of $v_{k+1}$. \eqref{eq:Ydiff-proof-split} follows from the inequality $\n{x+y}^{2}\leq2\n x^{2}+2\n y^{2}$. The rest is simple algebraic manipulation. 
\begin{align*}
\n{y_{k+1}-y_{k}}^{2} & \leq2\alpha^{2}\beta^{2}\n{v_{k}-y_{k}}^{2}+2L_{i_{k}}^{-1}\p{\alpha\sigma^{-1/2}+h\p{1-\alpha}L_{i_{k}}^{-1/2}}^{2}\n{\nabla_{i^{k}}f\p{\hat{y}_{k}}}^{2}\\
 \text{($\underbar{L}\leq L_i,\forall i$)}& \leq2\alpha^{2}\beta^{2}\n{v_{k}-y_{k}}^{2}+2L_{i_{k}}^{-1}\p{\alpha\sigma^{-1/2}+h\p{1-\alpha}\underbar{L}^{-1/2}}^{2}\n{\nabla_{i^{k}}f\p{\hat{y}_{k}}}^{2}\\
 & =2\alpha^{2}\beta^{2}\n{v_{k}-y_{k}}^{2}+2L_{i_{k}}^{-1}\underbar{L}^{-1}\p{\underbar{L}^{1/2}\sigma^{-1/2}\alpha+h\p{1-\alpha}}^{2}\n{\nabla_{i^{k}}f\p{\hat{y}_{k}}}^{2}\\
\EE\n{y_{k+1}-y_{k}}^{2} & \leq2\alpha^{2}\beta^{2}\n{v_{k}-y_{k}}^{2}+2S^{-1}\underbar{L}^{-1}\p{\underbar{L}^{1/2}\sigma^{-1/2}\alpha+h\p{1-\alpha}}^{2}\n{\nabla f\p{\hat{y}_{k}}}_{*}^{2}
\end{align*}

Finally, to complete the proof, we prove $\underbar{L}^{1/2}\sigma^{-1/2}\alpha+h\p{1-\alpha}\leq1$.
\begin{align}
  \underbar{L}^{1/2}\sigma^{-1/2}\alpha+h\p{1-\alpha} & =h+\alpha\p{\underbar{L}^{1/2}\sigma^{-1/2}-h}\nonumber\\
\text{(definitions of \ensuremath{h} and \ensuremath{\alpha}: \eqref{eq:def:alpha}, and \eqref{eq:def:h})} & =1-\frac{1}{2}\sigma^{1/2}\underbar{L}^{-1/2}\psi+\sigma^{1/2}S^{-1}\p{\underbar{L}^{1/2}\sigma^{-1/2}}\nonumber\\
 & \leq1-\sigma^{1/2}\underbar{L}^{-1/2}\p{\frac{1}{2}\psi-\sigma^{-1/2}S^{-1}\underbar{L}^{1}}\label{eq:Async-cost-proof-leq-1-inequality-1}
\end{align}
Rearranging the definition of $\psi$, we have:
\begin{align*}
S^{-1} & =\frac{1}{9^{2}}\psi^{2}\underbar{L}^{1}L^{-3/2}\kappa^{-1/2}\tau^{-2}\\
\text{(\ensuremath{\tau\geq}1 and \ensuremath{\psi\leq\frac{1}{2}})} & \leq\frac{1}{18^{2}}\underbar{L}^{1}L^{-3/2}\kappa^{-1/2}
\end{align*}
Using this on \eqref{eq:Async-cost-proof-leq-1-inequality-1}, we have:
\begin{align*}
\underbar{L}^{1/2}\alpha\sigma^{-1/2}+h\p{1-\alpha} & \leq1-\sigma^{1/2}\underbar{L}^{-1/2}\p{\frac{1}{2}\psi-\frac{1}{18^{2}}\underbar{L}^{1}L^{-3/2}\kappa^{-1/2}\sigma^{-1/2}\underbar{L}^{1}}\\
 & =1-\sigma^{1/2}\underbar{L}^{-1/2}\p{\frac{1}{2}\psi-\frac{1}{18^{2}}\p{\underbar{L}/L}^{2}}\\
\text{(\ensuremath{\psi\leq\frac{1}{2}})} & =1-\sigma^{1/2}\underbar{L}^{-1/2}\p{\frac{1}{24}-\frac{1}{18^{2}}}\leq1.
\end{align*}
This completes the proof.
\end{proof}

\subsection{Master inequality}
We are finally in a position to bring together all the all the previous results together into a master inequality for the Lyapunov function $\rho_{k}$ (defined in \eqref{eq:def:rho}). After this lemma is proven, we will prove that the right hand size is negative, which will imply that $\rho_{k}$ linearly converges to $0$ with rate $\beta$. 

\begin{lem}[Master inequality]
We have:\label{lem:Master-inequality}
\begin{align}
 &  &  & \EE_{k}\sp{\rho_{k+1}-\beta\rho_{k}}\nonumber \\
 &  & \leq & +\n{y_{k}}^{2} & \times & \p{1-\beta-\sigma^{-1/2}S^{-1}\sigma\p{1-\psi}}\label{eq:Master-inequality}\\
 &  &  & +\n{v_{k}-y_{k}}^{2} & \times & \beta\p{2\alpha^{2}\beta c_{1}+S^{-1}\beta L^{1/2}\kappa^{-1/2}\psi-\p{1-\beta}}\nonumber \\
 &  &  & +f\p{y_{k}} & \times & \p{c-2\sigma^{-1/2}S^{-1}\p{\beta\alpha^{-1}\p{1-\alpha}+1}}\nonumber \\
 &  &  & +f\p{x_{k}} & \times & \beta\p{2\sigma^{-1/2}S^{-1}\alpha^{-1}\p{1-\alpha}-c}\nonumber \\
 &  &  & +\sum_{j=1}^{\tau}\n{y_{k+1-j}-y_{k-j}}^{2} & \times & S^{-1}L\kappa\psi^{-1}\tau\sigma^{1/2}\p{2\sigma^{-1}+c}-s\nonumber \\
 &  &  & +\n{\nabla f\p{\hat{y}_{k}}}_{*}^{2} & \times & S^{-1}\p{\sigma^{-1}+2\underbar{L}^{-1}c_{1}-\frac{1}{2}ch\p{2-h\p{1+\frac{1}{2}\sigma^{1/2}\underbar{L}^{-1/2}\psi}}}\nonumber 
\end{align}

\end{lem}

\begin{proof}
\begin{align}
 & \EE_{k}\n{v_{k+1}}^{2}-\beta\n{v_{k}}^{2}\nonumber \\
\text{(\ref{eq:Starting-point})} & =\p{1-\beta}\n{y_{k}}^{2}-\beta\p{1-\beta}\n{v_{k}-y_{k}}^{2}+S^{-1}\sigma^{-1}\n{\nabla f\p{\hat{y}_{k}}}_{*}^{2}\nonumber \\
 & -2\sigma^{-1/2}S^{-1}\dotp{y_{k},\nabla f\p{\hat{y}_{k}}}\nonumber \\
 & -2\sigma^{-1/2}S^{-1}\beta\alpha^{-1}\p{1-\alpha}\dotp{y_{k}-x_{k},\nabla f\p{\hat{y}_{k}}}\nonumber \\
 & \leq\p{1-\beta}\n{y_{k}}^{2}-\beta\p{1-\beta}\n{v_{k}-y_{k}}^{2}+S^{-1}\sigma^{-1}\n{\nabla f\p{\hat{y}_{k}}}_{*}^{2}\label{eq:master-inequality-proof-vkp1}\\
\text{(\ref{eq:Cross-term-lemma-1})} & +2\sigma^{-1/2}S^{-1}\p{-f\p{y_{k}}-\frac{1}{2}\sigma\p{1-\psi}\n{y_{k}}^{2}+\frac{1}{2}L\kappa\psi^{-1}\tau\sum_{j=1}^{\tau}\n{y_{k+1-j}-y_{k-j}}^{2}}\nonumber \\
\text{(\ref{eq:Cross-term-lemma-2})} & -2\sigma^{-1/2}S^{-1}\beta\alpha^{-1}\p{1-\alpha}\p{f\p{x_{k}}-f\p{y_{k}}}\nonumber \\
 & +\sigma^{-1/2}S^{-1}\beta L\p{\kappa^{-1}\psi\beta\n{v_{k}-y_{k}}^{2}+\kappa\psi^{-1}\beta^{-1}\tau\sum_{j=1}^{\tau}\n{y_{k+1-j}-y_{k-j}}^{2}}\nonumber 
\end{align}
We now collect and organize the similar terms of this inequality.
\begin{align*}
&&\leq & +\n{y_{k}}^{2} & \times & \p{1-\beta-\sigma^{-1/2}S^{-1}\sigma\p{1-\psi}}\\
&& & +\n{v_{k}-y_{k}}^{2} & \times & \beta\p{\sigma^{-1/2}S^{-1}\beta L\kappa^{-1}\psi-\p{1-\beta}}\\
&& & -f\p{y_{k}} & \times & 2\sigma^{-1/2}S^{-1}\p{\beta\alpha^{-1}\p{1-\alpha}+1}\\
&& & +f\p{x_{k}} & \times & 2\sigma^{-1/2}S^{-1}\beta\alpha^{-1}\p{1-\alpha}\\
&& & +\sum_{j=1}^{\tau}\n{y_{k+1-j}-y_{k-j}}^{2} & \times & 2\sigma^{-1/2}S^{-1}L\kappa\psi^{-1}\tau\\
&& & +\n{\nabla f\p{\hat{y}_{k}}}_{*}^{2} & \times & \sigma^{-1}S^{-1}
\end{align*}
Now finally, we add the function-value and asynchronicity terms to
our analysis. We use Lemma \ref{lem:Asynchronicity-error-lemma} is
with $r=1-\sigma^{1/2}S^{-1}$, and 
\begin{align}
s_{i} & =\begin{cases}
s=6S^{-1}L^{1/2}\kappa^{3/2}\psi^{-1}\tau, & 1\leq i\leq\tau\\
0, & i>\tau
\end{cases}\label{eq:def:s}
\end{align}
Notice that this choice of $s_{i}$ will recover the coefficient formula
given in \eqref{eq:def:Asynchronicity-error-and-ci}. Hence we have:
\begin{align}
 & \EE_{k}\sp{cf\p{x_{k+1}}+A_{k+1}-\beta\p{cf\p{x_{k}}+A_{k}}}\nonumber \\
\text{(Lemma \ref{lem:Function-value-lemma})} & \leq cf\p{y_{k}}-\frac{1}{2}ch\p{2-h\p{1+\frac{1}{2}\sigma^{1/2}\underbar{L}^{-1/2}\psi}}S^{-1}\n{\nabla f\p{\hat{y}_{k}}}_{*}^{2}-\beta cf\p{x_{k}}\label{eq:master-inequality-proof-fA-kp1}\\
 & +S^{-1}L\sigma^{1/2}\kappa\psi^{-1}\tau\sum_{j=1}^{\tau}\n{y_{k+1-j}-y_{k-j}}^{2}\nonumber \\
\text{(Lemmas \ref{lem:Asynchronicity-error-lemma} and \ref{lem:y-Diff-lemma})} & +c_{1}\p{2\alpha^{2}\beta^{2}\n{v_{k}-y_{k}}^{2}+2S^{-1}\underbar{L}^{-1}\n{\nabla f\p{\hat{y}_{k}}}^{2}}\nonumber\\
&-\sum_{j=1}^{\infty}s_{j}\n{y_{k+1-j}-y_{k-j}}^{2}+A_{k}\p{r-\beta}\nonumber 
\end{align}
Notice $A_{k}\p{r-\beta}\leq0$. Finally, combining \eqref{eq:master-inequality-proof-vkp1}
and \eqref{eq:master-inequality-proof-fA-kp1} completes the proof.
\end{proof}
In the next section, we will prove that every coefficient on the right
hand side of \eqref{eq:Master-inequality} is $0$ or less, which
will complete the proof of Theorem \ref{thm:Main-theorem}.

\subsection{Proof of main theorem}
\begin{lem}
\textup{The coefficients of $\n{y_{k}}^{2}$, $f\p{y_{k}}$, and $\sum_{j=1}^{\tau}\n{y_{k+1-j}-y_{k-j}}^{2}$
in Lemma \ref{lem:Master-inequality} are non-positive.\label{lem:yk-fyk-diff-coeffs-negative}}
\end{lem}
\begin{proof}
The coefficient $1-\p{1-\psi}\sigma^{1/2}S^{-1}-\beta$ of $\n{y_{k}}^{2}$
is identically $0$ via the definition \eqref{eq:def:beta} of $\beta$.
The coefficient $c-2\sigma^{-1/2}S^{-1}\p{\beta\alpha^{-1}\p{1-\alpha}+1}$
of $f\p{y_{k}}$ is identically $0$ via the definition \eqref{eq:def:c}
of $c$. 

First notice from the definition \eqref{eq:def:c} of $c$:
\begin{align}
c & =2\sigma^{-1/2}S^{-1}\p{\beta\alpha^{-1}\p{1-\alpha}+1}\nonumber \\
\text{(definitions of \ensuremath{\alpha,\beta})} & =2\sigma^{-1/2}S^{-1}\p{\p{1-\sigma^{1/2}S^{-1}\p{1-\psi}}\p{1+\psi}\sigma^{-1/2}S+1}\nonumber \\
 & =2\sigma^{-1/2}S^{-1}\p{\p{1+\psi}\sigma^{-1/2}S+\psi^{2}}\nonumber \\
 & =2\sigma^{-1}\p{\p{1+\psi}+\psi^{2}\sigma^{1/2}S^{-1}}\label{eq:c-equals}\\
c & \leq4\sigma^{-1}\label{eq:c-leq}
\end{align}

Here the last line followed since $\psi\leq\frac{1}{2}$ and $\sigma^{1/2}S^{-1}\leq1$.
We now analyze the coefficient of $\sum_{j=1}^{\tau}\n{y_{k+1-j}-y_{k-j}}^{2}$.
\begin{align*}
 & S^{-1}L\kappa\psi^{-1}\tau\sigma^{1/2}\p{2\sigma^{-1}+c}-s\\
\text{(\ref{eq:c-leq})} & \leq6L^{1/2}\kappa^{3/2}\psi^{-1}\tau-s\\
\text{(definition \eqref{eq:def:s} of \ensuremath{s})} & \leq0\qedhere
\end{align*}
\end{proof}
\begin{lem}
\textup{The coefficient $\beta\p{2\sigma^{-1/2}S^{-1}\alpha^{-1}\p{1-\alpha}-c}$
of $f\p{x_{k}}$ in Lemma \ref{lem:Master-inequality} is non-positive.\label{lem:fxk-coeff-negative}}
\end{lem}
\begin{proof}
\begin{align*}
 & 2\sigma^{-1/2}S^{-1}\alpha^{-1}\p{1-\alpha}-c\\
\text{(\ref{eq:c-equals})} & =2\sigma^{-1/2}S^{-1}\p{1+\psi}\sigma^{-1/2}S-2\sigma^{-1}\p{\p{1+\psi}+\psi^{2}\sigma^{1/2}S^{-1}}\\
 & =2\sigma^{-1}\p{\p{1+\psi}-\p{\p{1+\psi}+\psi^{2}\sigma^{1/2}S^{-1}}}\\
 & =-2\psi^{2}\sigma^{-1/2}S^{-1}\leq0\qedhere
\end{align*}
 
\end{proof}
\begin{lem}
\textup{The coefficient $S^{-1}\p{\sigma^{-1}+2\underbar{L}^{-1}c_{1}-\frac{1}{2}ch\p{2-h\p{1+\frac{1}{2}\sigma^{1/2}\underbar{L}^{-1/2}\psi}}}$
of $\n{\nabla f\p{\hat{y}_{k}}}_{*}^{2}$ in Lemma \ref{lem:Master-inequality}
is non-positive.\label{lem:grad-coeff-negative}}
\end{lem}
\begin{proof}
We first need to bound $c_{1}$. 
\begin{align*}
\text{(\eqref{eq:def:s} and \eqref{eq:def:Asynchronicity-error-and-ci}) }\,c_{1} & =s\sum_{j=1}^{\tau}\p{1-\sigma^{1/2}S^{-1}}^{-j}\\
\text{\eqref{eq:def:s} } & \leq6S^{-1}L^{1/2}\kappa^{3/2}\psi^{-1}\tau\sum_{j=1}^{\tau}\p{1-\sigma^{1/2}S^{-1}}^{-j}\\
 & \leq6S^{-1}L^{1/2}\kappa^{3/2}\psi^{-1}\tau^{2}\p{1-\sigma^{1/2}S^{-1}}^{-\tau}
\end{align*}
It can be easily verified that if $x\leq\frac{1}{2}$ and $y\geq0$,
then $\p{1-x}^{-y}\leq\exp\p{2xy}$. Using this fact with $x=\sigma^{1/2}S^{-1}$
and $y=\tau$, we have:
\begin{align}
 & \leq6S^{-1}L^{1/2}\kappa^{3/2}\psi^{-1}\tau^{2}\exp\p{\tau\sigma^{1/2}S^{-1}}\nonumber \\
\text{(since $\psi\leq 3/7$ and hence \ensuremath{\tau\sigma^{1/2}S^{-1}\leq\frac{1}{7}})} & \leq S^{-1}L^{1/2}\kappa^{3/2}\psi^{-1}\tau^{2}\times6\exp\p{\frac{1}{7}}\nonumber \\
c_{1} & \leq7S^{-1}L^{1/2}\kappa^{3/2}\psi^{-1}\tau^{2}\label{eq:c1-leq}
\end{align}
We now analyze the coefficient of $\n{\nabla f\p{\hat{y}_{k}}}_{*}^{2}$
\begin{align*}
 & \sigma^{-1}+2\underbar{L}^{-1}c_{1}-\frac{1}{2}ch\p{2-h\p{1+\frac{1}{2}\sigma^{1/2}\underbar{L}^{-1/2}\psi}}\\
\text{(\ref{eq:c1-leq} and \ref{eq:def:h})} & \leq\sigma^{-1}+14S^{-1}\underbar{L}^{-1}L^{1/2}\kappa^{3/2}\psi^{-1}\tau^{2}-\frac{1}{2}ch\p{1+\frac{1}{4}\sigma^{1}\underbar{L}^{-1}\psi^{2}}\\
 & \leq\sigma^{-1}+14S^{-1}\underbar{L}^{-1}L^{1/2}\kappa^{3/2}\psi^{-1}\tau^{2}-\frac{1}{2}ch\\
\text{(definition \ref{eq:def:psi} of \ensuremath{\psi})} & =\sigma^{-1}+\frac{14}{81}\sigma^{-1}\psi-\frac{1}{2}ch\\
\text{(\ref{eq:c-equals}, definition \ref{eq:def:h} of \ensuremath{h})} & =\sigma^{-1}\p{1+\frac{14}{81}\psi-\p{\p{1+\psi}+\psi^{2}\sigma^{1/2}S^{-1}}\p{1-\frac{1}{2}\sigma^{1/2}\underbar{L}^{-1/2}\psi}}\\
\text{(\ensuremath{\sigma^{1/2}\underbar{L}^{-1/2}\leq0} and \ensuremath{\sigma^{1/2}S^{-1}\leq1})} & \leq\sigma^{-1}\p{1+\frac{14}{81}\psi-\p{1+\psi}\p{1-\frac{1}{2}\psi}}\\
 & =\sigma^{-1}\psi\p{\frac{14}{81}+\frac{1}{2}\psi-\frac{1}{2}}\\
\text{(\ensuremath{\psi\leq\frac{1}{2}})} & \leq0\qedhere
\end{align*}
\end{proof}
\begin{lem}
\textup{The coefficient $\beta\p{2\alpha^{2}\beta c_{1}+S^{-1}\beta L^{1/2}\kappa^{-1/2}\psi-\p{1-\beta}}$
of $\n{v_{k}-y_{k}}^{2}$ in \ref{lem:Master-inequality} is non-positive.\label{lem:vy-coeff-negative}}
\end{lem}
\begin{proof}
\begin{align*}
 & 2\alpha^{2}\beta c_{1}+\sigma^{1/2}S^{-1}\beta\psi-\p{1-\psi}\sigma^{1/2}S^{-1}\\
\text{(\ref{eq:c1-leq})} & \leq14\alpha^{2}\beta S^{-1}L^{1/2}\kappa^{3/2}\psi^{-1}\tau^{2}+\sigma^{1/2}S^{-1}\beta\psi-\p{1-\psi}\sigma^{1/2}S^{-1}\\
 & \leq14\sigma S^{-3}L^{1/2}\kappa^{3/2}\psi^{-1}\tau^{2}+\sigma^{1/2}S^{-1}\psi-\p{1-\psi}\sigma^{1/2}S^{-1}\\
 & =\sigma^{1/2}S^{-1}\p{14S^{-2}L\kappa\tau^{2}\psi^{-1}+2\psi-1}
\end{align*}
Here the last inequality follows since $\beta\leq1$ and $\alpha\leq\sigma^{1/2}S^{-1}$.
We now rearrange the definition of $\psi$ to yield the identity:
\begin{align*}
S^{-2}\kappa & =\frac{1}{9^{4}}\underbar{L}^{2}L^{-3}\tau^{-4}\psi^{4}
\end{align*}

Using this, we have:
\begin{align*}
 & 14S^{-2}L\kappa\tau^{2}\psi^{-1}+2\psi-1\\
 & =\frac{14}{9^{4}}\underbar{L}^{2}L^{-2}\psi^{3}\tau^{-2}+2\psi-1\\
 & \leq\frac{14}{9^{4}}\p{\frac{3}{7}}^{3}1^{-2}+\frac{6}{7}-1\leq0
\end{align*}
Here the last line followed since $\underbar{L}\leq L$, $\psi\leq\frac{3}{7}$,
and $\tau\geq1$. Hence the proof is complete.
\end{proof}
\begin{proof}[Proof of Theorem \ref{thm:Main-theorem}]
 Using the master inequality \ref{lem:Master-inequality} in combination
with the previous Lemmas \ref{lem:yk-fyk-diff-coeffs-negative}, \ref{lem:fxk-coeff-negative},
\ref{lem:grad-coeff-negative}, and \ref{lem:vy-coeff-negative},
we have:
\begin{align*}
\EE_{k}\sp{\rho_{k+1}} & \leq\beta\rho_{k}=\p{1-\p{1-\psi}\sigma^{1/2}S^{-1}}\rho_{k}
\end{align*}
When we have:
\begin{align*}
\p{1-\p{1-\psi}\sigma^{1/2}S^{-1}}^{k} & \leq\eps
\end{align*}
then the Lyapunov function $\rho_{k}$ has decreased below $\eps\rho_{0}$ in expectation.
Hence the complexity $K\p{\eps}$ satisfies:
\begin{align*}
K\p{\eps}\ln\p{1-\p{1-\psi}\sigma^{1/2}S^{-1}} & =\ln\p{\eps}\\
K\p{\eps} & =\frac{-1}{\ln\p{1-\p{1-\psi}\sigma^{1/2}S^{-1}}}\ln\p{1/\eps}
\end{align*}
Now it can be shown that for $0<x\leq\frac{1}{2}$, we have:
\begin{align*}
\frac{1}{x}-1 & \leq\frac{-1}{\ln\p{1-x}}\leq\frac{1}{x}-\frac{1}{2}\\
\frac{-1}{\ln\p{1-x}} & =\frac{1}{x}+\cO\p 1
\end{align*}
Since $n\geq2$, we have $\sigma^{1/2}S^{-1}\leq\frac{1}{2}$. Hence:
\begin{align*}
K\p{\eps} & =\frac{1}{1-\psi}\p{\sigma^{-1/2}S+\cO\p 1}\ln\p{1/\eps}
\end{align*}
An expression for $K_{\text{\text{{\tt NU\_ACDM}}}}\p{\eps}$, the complexity of $\text{{\tt NU\_ACDM}}$ follows by similar reasoning.
\begin{align}
K_{\text{{\tt NU\_ACDM}}}\p{\eps} & =\p{\sigma^{-1/2}S+\cO\p 1}\ln\p{1/\eps}
\end{align}
Finally we have:
\begin{align*}
K\p{\eps} & =\frac{1}{1-\psi}\p{\frac{\sigma^{-1/2}S+\cO\p 1}{\sigma^{-1/2}S+\cO\p 1}}K_{\text{\text{{\tt NU\_ACDM}}}}\p{\eps}\\
 & =\frac{1}{1-\psi}\p{1+o\p 1}K_{\text{\text{{\tt NU\_ACDM}}}}\p{\eps}
\end{align*}
which completes the proof.
\end{proof}

\section{Ordinary Differential Equation Analysis}

\subsection{Derivation of ODE for synchronous \texttt{A2BCD}\label{subsec:Derivation-of-ODE}}

If we take expectations with respect to $\EE_{k}$, then synchronous (no delay) $\algoNameS$
becomes:
\begin{align*}
y_{k} & =\alpha v_{k}+\p{1-\alpha}x_{k}\\
\EE_{k}x_{k+1} & =y_{k}-n^{-1}\kappa^{-1}\nabla f\p{y_{k}}\\
\EE_{k}v_{k+1} & =\beta v_{k}+\p{1-\beta}y_{k}-n^{-1}\kappa^{-1/2}\nabla f\p{y_{k}}
\end{align*}
We find it convenient to define $\eta=n\kappa^{1/2}$. Inspired by
this, we consider the following iteration:

\begin{align}
y_{k} & =\alpha v_{k}+\p{1-\alpha}x_{k}\label{eq:xk-def-ODE}\\
x_{k+1} & =y_{k}-s^{1/2}\kappa^{-1/2}\eta^{-1}\nabla f\p{y_{k}}\label{eq:yk-def-ODE}\\
v_{k+1} & =\beta v_{k}+\p{1-\beta}y_{k}-s^{1/2}\eta^{-1}\nabla f\p{y_{k}}\label{eq:vk-def-ODE}
\end{align}
for coefficients:
\begin{align*}
\alpha & =\p{1+s^{-1/2}\eta}^{-1}\\
\beta & =1-s^{1/2}\eta^{-1}
\end{align*}
$s$ is a discretization scale parameter that will be sent to $0$ to obtain an ODE
analogue of synchronous $\algoNameS$. We first use \eqref{eq:vk-in-terms-xk-yk}
to eliminate $v_{k}$ from from \eqref{eq:vk-def-ODE}.
\begin{align*}
0 & =-v_{k+1}+\beta v_{k}+\p{1-\beta}y_{k}-s^{1/2}\eta^{-1}\nabla f\p{y_{k}}\\
0 & =-\alpha^{-1}y_{k+1}+\alpha^{-1}\p{1-\alpha}x_{k+1}\\
 & +\beta\p{\alpha^{-1}y_{k}-\alpha^{-1}\p{1-\alpha}x_{k}}+\p{1-\beta}y_{k}-s^{1/2}\eta^{-1}\nabla f\p{y_{k}}\\
\text{(times by \ensuremath{\alpha}) }0 & =-y_{k+1}+\p{1-\alpha}x_{k+1}\\
 & +\beta\p{y_{k}-\p{1-\alpha}x_{k}}+\alpha\p{1-\beta}y_{k}-\alpha s^{1/2}\eta^{-1}\nabla f\p{y_{k}}\\
 & =-y_{k+1}+y_{k}\p{\beta+\alpha\p{1-\beta}}\\
 & +\p{1-\alpha}x_{k+1}-x_{k}\beta\p{1-\alpha}-\alpha s^{1/2}\eta^{-1}\nabla f\p{y_{k}}
\end{align*}
We now eliminate $x_{k}$ using \eqref{eq:xk-def-ODE}:
\begin{align*}
0 & =-y_{k+1}+y_{k}\p{\beta+\alpha\p{1-\beta}}\\
 & +\p{1-\alpha}\p{y_{k}-s^{1/2}\eta^{-1}\kappa^{-1/2}\nabla f\p{y_{k}}}-\p{y_{k-1}-s^{1/2}\eta^{-1}\kappa^{-1/2}\nabla f\p{y_{k-1}}}\beta\p{1-\alpha}\\
 & -\alpha s^{1/2}\eta^{-1}\nabla f\p{y_{k}}\\
 & =-y_{k+1}+y_{k}\p{\beta+\alpha\p{1-\beta}+\p{1-\alpha}}-\beta\p{1-\alpha}y_{k-1}\\
 & +s^{1/2}\eta^{-1}\nabla f\p{y_{k-1}}\p{\beta-1}\p{1-\alpha}\\
 & -\alpha s^{1/2}\eta^{-1}\nabla f\p{y_{k}}\\
 & =\p{y_{k}-y_{k+1}}+\beta\p{1-\alpha}\p{y_{k}-y_{k-1}}\\
 & +s^{1/2}\eta^{-1}\p{\nabla f\p{y_{k-1}}\p{\beta-1}\p{1-\alpha}-\alpha\nabla f\p{y_{k}}}
\end{align*}
Now to derive an ODE, we let $y_{k}=Y\p{ks^{1/2}}$. Then $\nabla f\p{y_{k-1}}=\nabla f\p{y_{k}}+\cO\p{s^{1/2}}$.
Hence the above becomes:
\begin{align}
0 & =\p{y_{k}-y_{k+1}}+\beta\p{1-\alpha}\p{y_{k}-y_{k-1}}\nonumber \\
 & +s^{1/2}\eta^{-1}\p{\p{\beta-1}\p{1-\alpha}-\alpha}\nabla f\p{y_{k}}+\cO\p{s^{3/2}}\nonumber \\
0 & =\p{-s^{1/2}\dot{Y}-\frac{1}{2}s\ddot{Y}}+\beta\p{1-\alpha}\p{s^{1/2}\dot{Y}-\frac{1}{2}s\ddot{Y}}\label{eq:Asymptotic-expansion-iteration-Y}\\
 & +s^{1/2}\eta^{-1}\p{\p{\beta-1}\p{1-\alpha}-\alpha}\nabla f\p{y_{k}}+\cO\p{s^{3/2}}\nonumber 
\end{align}
We now look at some of the terms in this equation to find the highest-order
dependence on $s$.
\begin{align*}
\beta\p{1-\alpha} & =\p{1-s^{1/2}\eta^{-1}}\p{1-\frac{1}{1+s^{-1/2}\eta}}\\
 & =\p{1-s^{1/2}\eta^{-1}}\frac{s^{-1/2}\eta}{1+s^{-1/2}\eta}\\
 & =\frac{s^{-1/2}\eta-1}{s^{-1/2}\eta+1}\\
 & =\frac{1-s^{1/2}\eta^{-1}}{1+s^{1/2}\eta^{-1}}\\
 & =1-2s^{1/2}\eta^{-1}+\cO\p s
\end{align*}
We also have:
\begin{align*}
\p{\beta-1}\p{1-\alpha}-\alpha & =\beta\p{1-\alpha}-1\\
 & =-2s^{1/2}\eta^{-1}+\cO\p s
\end{align*}
Hence using these facts on \eqref{eq:Asymptotic-expansion-iteration-Y},
we have:
\begin{align*}
0 & =\p{-s^{1/2}\dot{Y}-\frac{1}{2}s\ddot{Y}}+\p{1-2s^{1/2}\eta^{-1}+\cO\p s}\p{s^{1/2}\dot{Y}-\frac{1}{2}s\ddot{Y}}\\
 & +s^{1/2}\eta^{-1}\p{-2s^{1/2}\eta^{-1}+\cO\p s}\nabla f\p{y_{k}}+\cO\p{s^{3/2}}\\
0 & =-s^{1/2}\dot{Y}-\frac{1}{2}s\ddot{Y}+\p{s^{1/2}\dot{Y}-\frac{1}{2}s\ddot{Y}-2s^{1}\eta^{-1}\dot{Y}+\cO\p{s^{3/2}}}\\
 & \p{-2s^{1}\eta^{-2}+\cO\p{s^{3/2}}}\nabla f\p{y_{k}}+\cO\p{s^{3/2}}\\
0 & =-s\ddot{Y}-2s\eta^{-1}\dot{Y}-2s\eta^{-2}\nabla f\p{y_{k}}+\cO\p{s^{3/2}}\\
0 & =-\ddot{Y}-2\eta^{-1}\dot{Y}-2\eta^{-2}\nabla f\p{y_{k}}+\cO\p{s^{1/2}}
\end{align*}
Taking the limit as $s\to0$, we obtain the ODE:
\begin{align*}
\ddot{Y}\p t+2\eta^{-1}\dot{Y}+2\eta^{-2}\nabla f\p Y & =0
\end{align*}

\subsection{Convergence proof for synchronous ODE \label{subsec:Synchronous-ODE-proof}}

\begin{align*}
e^{-\eta^{-1}t}E'\p t & =\dotp{\nabla f\p{Y\p t},\dot{Y}\p t}+\eta^{-1}f\p{Y\p t}\\
 & +\frac{1}{2}\dotp{Y\p t+\eta \dot{Y}\p t,\dot{Y}\p t+\eta \ddot{Y}\p t}+\eta^{-1}\frac{1}{4}\n{Y\p t+\eta \dot{Y}\p t}^{2}\\
\text{(strong convexity \eqref{eq:Strongly-convex-bound})} & \leq\dotp{\nabla f\p Y,\dot{Y}}+\eta^{-1}\dotp{\nabla f\p Y,Y}-\frac{1}{2}\eta^{-1}\n Y^{2}\\
 & +\frac{1}{2}\dotp{Y+\eta \dot{Y},-\dot{Y}-2\eta^{-1}\nabla f\p Y}+\eta^{-1}\frac{1}{4}\n{Y\p t+\eta \dot{Y}\p t}^{2}\\
 & =-\frac{1}{4}\eta^{-1}\n Y^{2}-\frac{1}{4}\eta\n{\dot{Y}}^{2}\leq0
\end{align*}
Hence we have $E'(t)\leq 0$. Therefore $E(t)\leq E(0)$. That is:
\begin{align}
E(t) & =e^{n^{-1}\kappa^{-1/2}t}\p{f\p{Y}+\frac{1}{4}\n{Y+\eta \dot{Y}}^{2}}
 \leq E(0)
 = f\p{Y(0)}+\frac{1}{4}\n{Y(0)+\eta \dot{Y}(0)}^{2}
\end{align}
which implies:
\begin{align}
f\p{Y(t)}+\frac{1}{4}\n{Y(t)+\eta \dot{Y}(t)}^{2} &\leq e^{-n^{-1}\kappa^{-1/2}t} \p{ f\p{Y(0)}+\frac{1}{4}\n{Y(0)+\eta \dot{Y}(0)}^{2}}
\end{align}

\subsection{Asynchronicity error lemma\label{subsec:Asynchronicity-error-lemma}}
This result is the continuous-time analogue of Lemma \ref{lem:Asynchronicity-error-lemma}. First notice that $c\p 0=c_{0}$ and $c\p{\tau}=0$. We also have:
\begin{align*}
c'\p t/c_{0} & =-re^{-rt}-re^{-rt}\frac{e^{-r\tau}}{1-e^{-r\tau}}\\
 & =-r\p{e^{-rt}+e^{-rt}\frac{e^{-r\tau}}{1-e^{-r\tau}}}\\
 & =-r\p{e^{-rt}+\p{e^{-rt}-1}\frac{e^{-r\tau}}{1-e^{-r\tau}}+\frac{e^{-r\tau}}{1-e^{-r\tau}}}\\
c'\p t & =-rc\p t-rc_{0}\frac{e^{-r\tau}}{1-e^{-r\tau}}
\end{align*}
Hence using $c(\tau)=0$: 
\begin{align*}
A'\p t & =c_{0}\n{\dot{Y}\p t}^{2}+\int_{t-\tau}^{t}c'\p{t-s}\n{\dot{Y}\p s}^{2}ds\\
 & =c_{0}\n{\dot{Y}\p t}^{2}-rA\p t-rc_{0}\frac{e^{-r\tau}}{1-e^{-r\tau}}D\p t
\end{align*}
Now when $x\leq\frac{1}{2}$, we have $\frac{e^{-x}}{1-e^{-x}}\geq\frac{1}{2}x^{-1}$.
Hence when $r\tau\leq\frac{1}{2}$, we have:
\begin{align*}
A'\p t & \leq c_{0}\n{\dot{Y}\p t}^{2}-rA\p t-\frac{1}{2}\tau^{-1}c_{0}D\p t
\end{align*}
and the result easily follows.

\subsection{Convergence analysis for the asynchronous ODE\label{subsec:Async-ODE-proof}}
We consider the same energy as in the synchronous case (that is, the ODE in \eqref{eq:Acceleration-ODE}). Similar to
before, we have:
\begin{align*}
e^{-\eta^{-1}t}E'\p t & \leq\dotp{\nabla f\p Y,\dot{Y}}+\eta^{-1}\dotp{\nabla f\p Y,Y}-\frac{1}{2}\eta^{-1}\n Y^{2}\\
 & +\frac{1}{2}\dotp{Y+\eta \dot{Y},-\dot{Y}-2\eta^{-1}\nabla f\p{\hat{Y}}}+\eta^{-1}\frac{1}{4}\n{Y\p t+\eta \dot{Y}\p t}^{2}\\
 & =\dotp{\nabla f\p Y,\dot{Y}}+\eta^{-1}\dotp{\nabla f\p Y,Y}-\frac{1}{2}\eta^{-1}\n Y^{2}\\
 & +\frac{1}{2}\dotp{Y+\eta \dot{Y},-\dot{Y}-2\eta^{-1}\nabla f\p Y}+\eta^{-1}\frac{1}{4}\n{Y\p t+\eta \dot{Y}\p t}^{2}\\
 & -\eta^{-1}\dotp{Y+\eta \dot{Y},\nabla f\p{\hat{Y}}-\nabla f\p Y}\\
 & =-\frac{1}{4}\eta^{-1}\n Y^{2}-\frac{1}{4}\eta\n{\dot{Y}}^{2}-\eta^{-1}\dotp{Y+\eta \dot{Y},\nabla f\p{\hat{Y}}-\nabla f\p Y}
\end{align*}
where the final equality follows from the proof in Section \ref{subsec:Synchronous-ODE-proof}. Hence
\begin{align}
e^{-\eta^{-1}t}E'\p t & \leq-\frac{1}{4}\eta^{-1}\n Y^{2}-\frac{1}{4}\eta\n{\dot{Y}}^{2}+L\eta^{-1}\n Y\n{\hat{Y}-Y}+L\n{\dot{Y}}\n{\hat{Y}-Y}\label{eq:ODE-non-monotonic-energy}
\end{align}
Now we present an inequality that is similar to \eqref{lem:Async-difference-lemma}. 
\begin{lem}
Let $A,\chi>0$. Then:
\begin{align*}
\n{Y\p t-\hat{Y}\p t}A & \leq\frac{1}{2}\chi\tau D\p t+\frac{1}{2}\chi^{-1}A^{2}
\end{align*}
\end{lem}
\begin{proof}
Since $\hat{Y}(t)$ is a delayed version of $Y(t)$, we have: $\hat{Y}(t)=Y(t-j(t))$ for some function $j(t)\geq 0$ (though this can be easily generalized to an inconsistent read). Recall that for $\chi>0$, we have $ab\leq \frac{1}{2}\p{\chi a^2 +\chi^{-1} b^2}$. Hence
\begin{align*}
X\p t-\hat{X}\p t & =\int_{s=t-j\p t}^{t}X'\p sds\\
\n{X\p t-\hat{X}\p t}A & =\n{\int_{s=t-j\p t}^{t}X'\p sds}A\\
 & \leq\frac{1}{2}\chi\n{\int_{s=t-j\p t}^{t}X'\p sds}^{2}+\frac{1}{2}\chi^{-1}A^{2}\\
 \text{(Holder's inequality)}& \leq\frac{1}{2}\chi\p{\int_{s=t-j\p t}^{t}\n{X'\p s}^{2}ds}\p{\int_{s=t-j\p t}^{t}1ds}+\frac{1}{2}\chi^{-1}A^{2}\\
 & \leq\frac{1}{2}\chi\tau\p{\int_{s=t-j\p t}^{t}\n{X'\p s}^{2}ds}+\frac{1}{2}\chi^{-1}A^{2}
\end{align*}
\end{proof}
We use this lemma twice on $\n Y\n{\hat{Y}-Y}$ and $\n{\dot{Y}}\n{\hat{Y}-Y}$ in \eqref{eq:ODE-non-monotonic-energy}
with $\chi=2L,A=\n{Y}$ and $\chi=4L\eta^{-1},A=\n{\dot{Y}}$ respectively, to yield:
\begin{align*}
e^{-\eta^{-1}t}E'\p t & \leq-\frac{1}{4}\eta^{-1}\n Y^{2}-\frac{1}{4}\eta\n{\dot{Y}}^{2}\\
 & +L\eta^{-1}\p{L\tau D\p t+\frac{1}{4}L^{-1}\n Y^{2}}+L\p{2L\eta^{-1}\tau D\p t+\frac{1}{8}L^{-1}\eta\n{\dot{Y}}^{2}}\\
 & =-\frac{1}{8}\eta\n{\dot{Y}}^{2}+3L^{2}\eta^{-1}\tau D\p t
\end{align*}

The proof of convergence is completed in Section \ref{subsec:ODE-Analysis}.

\section{Optimality proof}\label{sec:Optimality-proof}
For parameter set $\sigma,L_{1},\ldots,L_{n},n$, we construct a block-separable function $f$ on the space $\RR^{bn}$ (separated into $n$ blocks of size $b$), which will imply this lower bound. Define $\kappa_{i}=L_i/\sigma$. We define the matrix $A_i\in\RR^{b\times b}$ via:
\begin{align*}
A_{i} & \triangleq\p{\begin{array}{ccccc}
2 & -1 & 0\\
-1 & 2 & \ddots & \ddots\\
0 & \ddots & \ddots & -1 & 0\\
 & \ddots & -1 & 2 & -1\\
 &  & 0 & -1 & \theta_{i}
\end{array}},\text{ for }\theta_{i}=\frac{\kappa_{i}^{1/2}+3}{\kappa_{i}^{1/2}+1}.
\end{align*}
Hence we define $f_{i}$ on $\RR^{b}$ via:
\begin{align*}
f_{i} & =\frac{L_{i}-\sigma}{4}\p{\frac{1}{2}\dotp{x,A_ix}-\dotp{e_{1},x}}+\frac{\sigma}{2}\n x^{2}
\end{align*}
which is clearly $\sigma$-strongly convex and $L_{i}$-Lipschitz on $\RR^b$.
From Lemma 8 of \parencite{LanZhou2017_optimal}, we know that this function
has unique minimizer 
\begin{align*}
x_{*,\p i} & \triangleq\p{q_{i},q_{i}^{2},\ldots,q_{i}^{b}},\text{ for }q=\frac{\kappa_{i}^{1/2}-1}{\kappa_{i}^{1/2}+1}.
\end{align*}
Finally, we define $f$ via:
\begin{align*}
f\p x & \triangleq\sum_{i=1}^{n}f_{i}\p{x_{\p i}}.
\end{align*}

Now let $e\p{i,j}$ be the $j$th unit vector of the $i$th block of size $b$ in $\RR^{bn}$. For $I_1,\ldots,I_n\in\NN$, we define the subspaces 
\begin{align*}
V_{i}\p I & =\text{span}\cp{e\p{i,1},\ldots,e\p{i,I}},\\
V\p{I_{1},\ldots,I_{n}} & =V_{1}\p{I_{1}}\oplus\ldots\oplus V_{n}\p{I_{n}}.
\end{align*}
$V\p{I_{1},\ldots,I_{n}}$ is the subspace with the first $I_{1}$ components of block $1$ nonzero, the first $I_{2}$ components of block $2$ nonzero, etc. First notice that $\text{IC}\p{V\p{I_{1},\ldots,I_{n}}}=V\p{I_{1},\ldots,I_{n}}$. Also, clearly, we have:
\begin{align}
\nabla_{i}f\p{V\p{I_{1},\ldots,I_{n}}} & \subset V\p{0,\ldots,0,\min\cp{I_{i}+1,b},0,\ldots,0}\label{eq:Lower-bound-subspace-transition}.
\end{align}
$\nabla_{i}f$ is supported on the $i$th block, hence why all the other indices are $0$. The patten of nonzeros in $A$ means that the gradient will have at most $1$ more nonzero on the $i$th block (see \parencite{Nesterov2013_introductory}). 

Let the initial point $x_{0}$ belong to $V\p{\bar{I_{1}},\ldots,\bar{I_{n}}}$. Let $I_{K,i}$ be the number of times we have had $i_{k}=i$ for $k=0,\ldots,K-1$. By induction on \eqref{eq:xkp1-in-IC-span} using \eqref{eq:Lower-bound-subspace-transition}, we have:
\begin{align*}
x_{k} & \in V\p{\min\cp{\bar{I_{1}}+I_{k,1},b},\ldots,\min\cp{\bar{I_{n}}+I_{k,m},b}}
\end{align*}
Hence if $x_{0,\p i}\in V_{i}\p 0$ and $k\leq b$, then
\begin{align*}
\n{x_{k,\p i}-x_{*,\p i}}^{2} & \geq\min_{x\in V_{i}\p{I_{k,i}}}\n{x-x_{*,\p i}}^{2} =\sum_{j=I_{k,i}+1}^{b}q_{i}^{2j} =\p{q_{i}^{2I_{k,i}+2}-q_{i}^{2b+2}}/\p{1-q_{i}^{2}}
\end{align*}

Therefore for all $i$, we have:
\begin{align*}
\EE\n{x_{k}-x_{*}}^{2} & \geq\EE\n{x_{k,\p i}-x_{*,\p i}}^{2} \geq\EE\sp{\p{q_{i}^{2I_{k,i}+2}-q_{i}^{2b+2}}/\p{1-q_{i}^{2}}}
\end{align*}
To evaluate this expectation, we note:
\begin{align*}
\EE q_{i}^{2I_{k,i}} & =\sum_{j=0}^{k}\p{\begin{array}{c}
k\\
j
\end{array}}p_{i}^{j}\p{1-p_{i}}^{k-j}q_{i}^{2j}\\
 & =\p{1-p_{i}}^{k}\sum_{j=0}^{k}\p{\begin{array}{c}
k\\
j
\end{array}}\p{q_{i}^{2}p_{i}\p{1-p_{i}}^{-1}}^{j}\\
 & =\p{1-p_{i}}^{k}\p{1+q_{i}^{2}p_{i}\p{1-p_{i}}^{-1}}^{k}\\
 & =\p{1-\p{1-q_{i}^{2}}p_{i}}^{k}
\end{align*}
Hence
\begin{align*}
\EE\n{x_{k}-x_{*}}^{2} & \geq\p{\p{1-\p{1-q_{i}^{2}}p_{i}}^{k}-q_{i}^{2b}}q_{i}^{2}/\p{1-q_{i}^{2}}.
\end{align*}
For any $i$, we may select the starting iterate $x_{0}$ by defining its block $j=1,\ldots,n$ via:
\begin{align*}
x_{0,\p j} & =\p{1-\del_{ij}}x_{*,\p j}
\end{align*}
For such a choice of $x_{0}$, we have
\begin{align*}
\n{x_{0}-x_{*}}^{2} & =\n{x_{*,\p i}}^{2} =q_{i}^{2}+\ldots+q_{i}^{2b}=q_{i}^{2}\frac{1-q_{i}^{2b}}{1-q_{i}^{2}}
\end{align*}
Hence for this choice of $x_0$:
\begin{align*}
\EE\n{x_{k}-x_{*}}^{2}/\n{x_{0}-x_{*}}^{2} & \geq\p{\p{1-\p{1-q_{i}^{2}}p_{i}}^{k}-q_{i}^{2b}}/\p{1-q_{i}^{2b}}
\end{align*}

Now notice:
\begin{align*}
\p{1-\p{1-q_{i}^{2}}p_{i}}^{k} & =\p{q_{i}^{-2}-\p{q_{i}^{-2}-1}p_{i}}^{k}q_{i}^{2k} \geq q_{i}^{2k}
\end{align*}
Hence
\begin{align*}
\EE\n{x_{k}-x_{*}}^{2}/\n{x_{0}-x_{*}}^{2} & \geq\p{1-\p{1-q_{i}^{2}}p_{i}}^{k}\p{1-q_{i}^{2b-2k}}/\p{1-q_{i}^{2b}}
\end{align*}
Now if we let $b=2k$, then we have:
\begin{align*}
\EE\n{x_{k}-x_{*}}^{2}/\n{x_{0}-x_{*}}^{2} & \geq\p{1-\p{1-q_{i}^{2}}p_{i}}^{k}\p{1-q_{i}^{2k}}/\p{1-q_{i}^{4k}}\\
 & =\p{1-\p{1-q_{i}^{2}}p_{i}}^{k}/\p{1+q_{i}^{2k}}\\
\EE\n{x_{k}-x_{*}}^{2}/\n{x_{0}-x_{*}}^{2} & \geq\frac{1}{2}\max_{i}\p{1-\p{1-q_{i}^{2}}p_{i}}^{k}
\end{align*}
Now let us take the minimum of the right-hand side over the parameters
$p_{i}$, subject to $\sum_{i=1}^{n}p_{i}=1$. The solution to this
minimization is clearly:
\begin{align*}
p_{i} & =\p{1-q_{i}^{2}}^{-1}/\p{\sum_{j=1}^{n}\p{1-q_{j}^{2}}^{-1}}
\end{align*}
Hence 
\begin{align*}
\EE\n{x_{k}-x_{*}}^{2}/\n{x_{0}-x_{*}}^{2} & \geq\frac{1}{2}\p{1-\p{\sum_{j=1}^{n}\p{1-q_{j}^{2}}^{-1}}^{-1}}^{k}\\
\sum_{j=1}^{n}\p{1-q_{j}^{2}}^{-1} & =\frac{1}{4}\sum_{j=1}^{n}\p{\kappa_{i}^{1/2}+2+\kappa_{i}^{-1/2}}\\
 & \geq\frac{1}{4}\p{\sum_{j=1}^{n}\kappa_{i}^{1/2}+2n}\\
\EE\n{x_{k}-x_{*}}^{2}/\n{x_{0}-x_{*}}^{2} & \geq\frac{1}{2}\p{1-\frac{4}{\sum_{j=1}^{n}\kappa_{i}^{1/2}+2n}}^{k}
\end{align*}
Hence the complexity $I\p{\eps}$ satisfies:
\begin{align*}
\eps & \geq\frac{1}{2}\p{1-\frac{4}{\sum_{j=1}^{n}\kappa_{i}^{1/2}+2n}}^{I\p{\eps}}\\
I\p{\eps} & \geq-\p{\ln\p{1-\frac{4}{\sum_{j=1}^{n}\kappa_{i}^{1/2}+2n}}}^{-1}\ln\p{1/2\eps}\\
 & =\frac{1}{4}\p{1+o\p 1}\p{n+\sum_{j=1}^{n}\kappa_{i}^{1/2}}\ln\p{1/2\eps}
\end{align*}

\section{Extensions}\label{sec:Extensions}
As mentioned, a stronger result than Theorem \ref{thm:Main-theorem} is possible. In the case when $L_i=L$ for all $i$, we can consider a slight modification of the coefficients:
\begin{align}
\alpha & \triangleq \p{1+\p{1+\psi}\sigma^{-1/2}S}^{-1}\\
\beta & \triangleq 1-\p{1+\psi}^{-1}\sigma^{1/2}S^{-1}\\
h & \triangleq \p{1+\frac{1}{2}\sigma^{1/2}L^{-1/2}\psi}^{-1}.
\end{align}
for the asynchronicity parameter:
\begin{align}
\psi &= 6\kappa^{1/2}n^{-1}\times \tau
\end{align}

This leads to complexity:
\begin{align}
K\p{\epsilon} &= \p{1+\psi}n\kappa^{1/2}\ln\p{1/\eps}
\end{align}
Here there is no restriction on $\psi$ as in Theorem \ref{thm:Main-theorem}, and hence there is no restriction on $\tau$. Assuming $\psi\leq 1$ gives optimal complexity to within a constant factor. Notice then that the resulting condition of $\tau$
\begin{align}
\tau &\leq \frac{1}{6}n\kappa^{-1/2}
\end{align} 
now essentially matches the one in Theorem \ref{thm:Async-ODE-Convergence} in Section \ref{subsec:ODE-Analysis}. While this result is stronger, it increases the complexity of the proof substantially. So in the interests of space and simplicity, we do not prove this stronger result.

\section{Efficient Implementation} \label{sec:Efficient implementation}
As mentioned in Section \ref{subsec:Numerical-experiments}, authors in \parencite{LeeSidford2013_efficienta} proposed a linear transformation of an accelerated RBCD scheme that results in sparse coordinate updates. Our proposed algorithm can be given a similar efficient implementation. We may eliminate $x^{k}$ from $\texttt{A2BCD}$, and derive the equivalent iteration below:
\begin{align*}
\p{\begin{array}{c}
y_{k+1}\\
v_{k+1}
\end{array}} & =\p{\begin{array}{cc}
1-\alpha\beta, & \alpha\beta\\
1-\beta, & \beta
\end{array}}\p{\begin{array}{c}
y_{k}\\
v_{k}
\end{array}}-\p{\begin{array}{c}
\p{\alpha\sigma^{-1/2}L_{i_{k}}^{-1/2}+h\p{1-\alpha}L_{i_{k}}^{-1}}\nabla_{i^{k}}f\p{\hat{y}^{k}}\\
\p{\sigma^{-1/2}L_{i_{k}}^{-1/2}}\nabla_{i^{k}}f\p{\hat{y}^{k}}
\end{array}}\\
 & \triangleq C\p{\begin{array}{c}
y_{k}\\
v_{k}
\end{array}}-Q_{k}
\end{align*}
where $C$ and $Q_{k}$ are defined in the obvious way. Hence we define auxiliary variables $p_{k},q_{k}$ defined via:
\begin{align}
\p{\begin{array}{c}
y_{k}\\
v_{k}
\end{array}} & =C^{k}\p{\begin{array}{c}
p_{k}\\
q_{k}
\end{array}}\label{eq:Relation-yv-pq}
\end{align}
These clearly follow the iteration:
\begin{align}
\p{\begin{array}{c}
p_{k+1}\\
q_{k+1}
\end{array}} & =\p{\begin{array}{c}
p_{k}\\
q_{k}
\end{array}}-C^{-\p{k+1}}Q_{k}\label{eq:Efficient-AARBCD-1-2}
\end{align}
Since the vector $Q_{k}$ is sparse, we can evolve variables $p_{k}$, and $q_{k}$ in a sparse manner, and recover the original iteration variables at the end of the algorithm via \ref{eq:Relation-yv-pq}. 

The gradient of the dual function is given by:
\begin{align*}
\nabla D\p y & =\frac{1}{\lambda d}\p{\frac{1}{d}A^{T}Ay+\lambda\p{y+l}}
\end{align*}
As mentioned before, it is necessary to maintain or recover $Ay^{k}$
to calculate block gradients. Since $Ay^{k}$ can be recovered via the linear relation in \eqref{eq:Relation-yv-pq}, and the gradient is an affine function, we maintain the auxiliary vectors $Ap^{k}$ and $Aq^{k}$ instead. 

Hence we propose the following efficient implementation in Algorithm \ref{alg:A2BCD}. We used this to generate the results in Table \ref{tab:Error-vs-time}. We also note also that it
can improve performance to periodically recover $v^{k}$ and $y^{k}$, reset the values of $p^{k}$, $q^{k}$, and $C$ to $v^{k}$, $y^{k}$, and $I$ respectively, and restarting the scheme (which can be done cheaply in time $\cO\p d$).

We let $B\in\RR^{2\times2}$ represent $C^{k}$, and $b$ represent $B^{-1}$. $\otimes$ is the Kronecker product. Each computing node has local outdated versions of $p, q, Ap, Aq$ which we denote $\hat p, \hat q, \hat{Ap}, \hat{Aq}$ respectively. We also find it convenient to define:
\begin{align}
\begin{bmatrix}
D_1^k \\ D_2^k
\end{bmatrix}
=
\begin{bmatrix}
\alpha\sigma^{-1/2}L_{i_{k}}^{-1/2}+h\p{1-\alpha}L_{i_{k}}^{-1} \\ \sigma^{-1/2}L_{i_{k}}^{-1/2}
\end{bmatrix}
\end{align}

\begin{algorithm}[H]
\caption{Shared-memory implementation of A2BCD} \label{alg:A2BCD}
\begin{algorithmic}[1] 
\State \textbf{Inputs:} Function parameters $A$, $\lambda$, $L$, $\cp{L_i}_{i=1}^{n}$, $n$, $d$. Delay $\tau$ (obtained in dry run). Starting vectors $y$, $v$. 
\State \textbf{Shared data:} Solution vectors $p$, $q$; auxiliary vectors $Ap$, $Aq$; sparsifying matrix $B$
\State \textbf{Node local data:} Solution vectors $\hat{p}$, $\hat{q}$, auxiliary vectors $\hat{Ap}$, $\hat{Aq}$, sparsifying matrix $\hat{B}$.
\State Calculate parameters $\psi, \alpha, \beta, h$ via \ref{def:AANRBCD}. Set $k=0$. 
\State \textbf{Initializations:} $p\gets y$, $q\gets v$, $Ap\gets Ay$, $Aq\gets Av$, $B\gets I$.
\While{not converged, each computing node asynchronous} 
	\State Randomly select block $i$ via \eqref{eq:Block-probability-distribution}.
	\State Read shared data into local memory: $\hat{p}\gets p$, $\hat{q}\gets q$, $\hat{Ap}\gets Ap$, $\hat{Aq}\gets Aq$, $\hat{B}\gets B$.
	\State Compute block gradient: $\nabla_i f(\hat{y})=\frac{1}{n\lambda}\p{\frac{1}{n}A_{i}^T\p{\hat{B}_{1,1}\hat{Ap}+\hat{B}_{1,2}\hat{Aq}}+\lambda\p{\hat{B}_{1,1}\hat{p}+\hat{B}_{1,2}\hat{q}}}$ 
	\State Compute quantity $g_i=A_i^T\nabla_i f(\hat{y})$
	\Statex \textbf{Shared memory updates:}
	\State Update $B\gets \begin{bmatrix}
1-\alpha\beta& \alpha\beta\\
1-\beta& \beta
\end{bmatrix}\times B$, calculate inverse $b\gets B^{-1}$.
	\State $\begin{array}{cl}
\begin{bmatrix}
p \\ q
\end{bmatrix} & \mathrel{-}= b\begin{bmatrix}
D_1^k \\ D_2^k
\end{bmatrix}\otimes \nabla_i f(\hat{y})
\end{array}$, 
$\quad\begin{array}{cl}
\begin{bmatrix}
Ap \\ Aq
\end{bmatrix} & \mathrel{-}= b\begin{bmatrix}
D_1^k \\ D_2^k
\end{bmatrix}\otimes g_i
\end{array}$
\State Increase iteration count: $k\gets k+1$
\EndWhile
\State Recover original iteration variables: $\begin{bmatrix} y \\ v \end{bmatrix} \gets B\begin{bmatrix} p \\ q \end{bmatrix}$. Output $y$. 
\end{algorithmic} 
\end{algorithm} 

\end{document}
